\documentclass[a4paper,10pt]{article}
\usepackage[utf8]{inputenc}
\usepackage{amsmath, amsthm, amssymb}
\usepackage{pstricks,pst-node,pst-plot}
\usepackage{enumerate}
\usepackage{booktabs}
\usepackage{todonotes}
\reversemarginpar

\let\eps\varepsilon 
\renewcommand{\epsilon}{\varepsilon}

\makeatletter
\def\@cite#1#2{{\normalfont[{\bfseries#1\if@tempswa , #2\fi}]}}
\makeatother

\title{Partitioning two-coloured complete multipartite graphs into monochromatic paths and cycles}
\author{Oliver Schaudt\\
Institut f\"ur Informatik, Universit\"at zu K\"oln, K\"oln, Germany\\
\texttt{schaudto@uni-koeln.de}\\~\\
Maya Stein\thanks{Supported by Fondecyt Regular no.~1140766.}\\
Centro de Modelamiento Matem\'atico, Universidad de Chile,\\ 
Santiago, Chile\\
\texttt{mstein@dim.uchile.cl}}

\begin{document}

\newtheorem{theorem}{Theorem}[section]
\newtheorem{lemma}[theorem]{Lemma}
\newtheorem{observation}[theorem]{Observation}
\newtheorem{corollary}[theorem]{Corollary}
\newtheorem{claim}[theorem]{Claim}
\newtheorem{conjecture}[theorem]{Conjecture}
\newtheorem{problem}[theorem]{Problem}

\maketitle

\begin{abstract}
\noindent 
We show that any complete $k$-partite graph $G$ on $n$ vertices, with $k \ge 3$, whose edges are two-coloured, can be covered with two vertex-disjoint monochromatic paths of distinct colours.
We prove this under the necessary assumption that the largest partition class of $G$ contains at most $n/2$ vertices.
This extends known results for complete and complete bipartite graphs.

Secondly, we show that in the same situation, all but $o(n)$ vertices of the graph can be covered with two vertex-disjoint monochromatic cycles of distinct colours, if colourings close to a split colouring are excluded.
From this we derive that the whole graph, if large enough, may be covered with 14 vertex-disjoint monochromatic cycles.

\noindent \textbf{keywords:} monochromatic path partition, monochromatic cycle partition, two-coloured graph

\end{abstract}

\section{Introduction}
\subsection{State of the art}

Gerencs\'er and Gy\'arf\'as~\cite{GG67} observed the vertex set of any complete graph whose edges are coloured red and blue\footnote{Note that a colouring is never meant to be a proper colouring in this paper: any assignment of colours red any blue to the edges will do.} can be partitioned into a red and a blue path. This is fairly easy: just take a maximal set $S$ of vertices that span two paths $P_1$, $P_2$, one in each colour, which only meet in one of their endvertices, call this vertex~$x$. One quickly checks that any vertex $v\notin S$ can be used to augment $S$: we can add the edge $xv$ to the path $P_i$ of the same colour, and then go from $v$ on reversely through $P_{3-i}$.
It is a long-standing conjecture that this phenomenon carries over to arbitrarily many colours.

\begin{conjecture}[Gy\'arf\'as~\cite{Gya89}]\label{con:gyarfas}
Let $G$ be a complete graph whose edges are coloured with $r$ colours.
Then $G$ can be partitioned into $r$ monochromatic paths.
\end{conjecture}

A stronger conjecture, replacing paths by cycles, had been put forward by Erd\H os, Gy\'arf\'as and Pyber~\cite{EGP91}, but was recently disproved by Pokrovskiy~\cite{Pok14} for $r\geq 3$. 
(Here, and throughout the paper, a cycle is allowed to consist of a single vertex or an edge, or to be totally empty.)
In the case $r=2$, however, the stronger result with cycles does hold (even with cycles of distinct colours). 
This used to be known as Lehel's conjecture, and was shown for all $n$ by Bessy and Thomass\'e~\cite{BT10}, after having been proved for large values of $n$ by \L uczak, R\"odl and Szemer\'edi~\cite{LRS98} and by Allen~\cite{A08}.

\begin{theorem}[Bessy and Thomass\'e~\cite{BT10}]\label{thm:thomasse}
Let $G$ be a complete graph whose edges are coloured red and blue.
Then $G$ can be partitioned into a red and a blue cycle.
\end{theorem}

Together with Conlon, the second author showed in~\cite{CS14} that Theorem~\ref{thm:thomasse} literally extends to $2$-local colourings: 
those are colourings with any number of colours, where each vertex is incident with at most two colours.

For arbitrary $r$, the best known bound on the number of vertex-disjoint cycles needed to cover the $r$-coloured complete graph $K_n$ is $100r\log r$,  if $n$ is large, this bound is due to Gy\'arf\'as, Ruszink\'o, S\'ark\"ozy and Szemer\'edi~\cite{GRSS06}.
  For $r=3$, the same authors show in~\cite{GRSS11} that  there is a partition of all but $o(n)$ vertices of $K_{n}$ into 3 or less monochromatic cycles. From this they deduce that  $17$ cycles partition the whole graph. 

\medskip

If one aims for similar results in complete bipartite graphs, it is reasonable to assume these are \emph{balanced}, i.e.~the two partition classes have the same size.
As observed by several authors, an obstruction for partitions of two-coloured balanced complete bipartite graphs into two paths/cycles is a certain type of colouring, which we will now describe. 

For any bipartite graph with partition classes $U$, $V$, call a red/blue colouring of $E(G)$ a \emph{split colouring} if there are partitions $U=A \cup B $ and $V=C \cup D$ such that all edges in $E_G(A,C) \cup E_G(B,D)$ are blue, and all edges in $E_G(A,D) \cup E_G(B,C)$ are red.
A split colouring is \emph{proper} if $\min\{|A|-|C|,|B|-|D|\} \ge 2$.

It is easy to see that a balanced complete bipartite graph with a proper split colouring cannot be partitioned into two monochromatic paths (even if the paths are allowed to have the same colour). That the converse is also true was shown by
Pokrovskiy~\cite{Pok14}, improving an earlier result of Gy\'arf\'as and Lehel~\cite{Gya83, GL73} (they allowed one uncovered vertex).

\begin{theorem}[Pokrovskiy~\cite{Pok14}]\label{thm:pokrovskiy}
Let $G$ be a balanced complete bipartite graph whose edges are coloured red and blue. If the colouring is not a proper split colouring,
then $G$ can be partitioned into a red and a blue path. 
\end{theorem}

It is not difficult to check that the vertices of any balanced complete bipartite graph with a proper split colouring can be partitioned into three monochromatic paths, or cycles.

 Haxell~\cite{Hax97} proved that any $r$-edge coloured balanced complete bipartite graph can be partitioned into $O((r\log r)^2)$ monochromatic cycles, and if $r=3$, then 1695 monochromatic cycles suffice. In~\cite{LSS14} we improve this number to 18.
Finally, we mention that  S\'ark\H ozy~\cite{Sar11} conjectured that any $r$-edge coloured graph $G$ can be partitioned into $r\alpha (G)$ monochromatic cycles: though this is false by Pokrovskiy's counterexample mentioned above, it is asymptotically true as shown by Balogh et al.~\cite{BBGGS}.

\subsection{Our Contribution}

We investigate monochromatic path or cycle partitions into complete multipartite graphs with more than two partition classes. We always assume that none of the partition classes is empty, {e.g.}, a complete tripartite graph is assumed to be non bipartite.
We call a multipartite graph \emph{fair} if no partition class contains more than half of the vertices of the graph. 
Note that a complete bipartite graph is fair if and only if it is balanced, and, more generally, a complete multipartite graph is fair if and only if it admits a Hamilton cycle.

We show an extension of Theorem~\ref{thm:pokrovskiy} to multipartite graphs with more than two partition classes under the necessary restriction of fairness. 
It is interesting that for these graphs,  there is no analogue for the exceptional case of the split colouring.

\begin{theorem}\label{thm:mainresult}
Let $G$ be a fair complete $k$-partite graph, with $k\geq 3$,
whose edges are coloured red and blue. Then $G$ can be partitioned into a red and a blue path. 
\end{theorem}

It seems plausible that Theorem~\ref{thm:mainresult} can be strengthened such that instead of two paths one can partition $G$ into a path and a cycle.
Indeed, we shall see this is true if we allow for at most one vertex to be uncovered.

\begin{corollary}\label{cor:path+cycle}
Let $G$ be a fair complete $k$-partite graph, with $k\geq 3$,
whose edges are coloured red and blue. Then all but at most one vertex of $G$ can be partitioned into a monochromatic path and a monochromatic cycle of distinct colours. 
\end{corollary}

We prove Theorem~\ref{thm:mainresult} and Corollary~\ref{cor:path+cycle} in Section~\ref{sec:paths}. The main tool for this proof is Lemma~\ref{lem:veryfair}, which is shown in Section~\ref{sec:uglylemmaproof}.

\medskip

It is natural to ask whether Theorems~\ref{thm:pokrovskiy} and~\ref{thm:mainresult} extend to cycle partitions instead of path partitions.
Note that we need to exclude the situation that there is a proper split colouring between a partition class that contains half the vertices of the graph, and the rest of the graph. Clearly, in that case a partition into two monochromatic cycles cannot exist (while a partition into two monochromatic paths is possible, by Theorem~\ref{thm:mainresult}).

We show an approximate result for the cycle partition problem in complete multipartite graphs, including the bipartite case.
For this, we say that a colouring of the edges of a complete multipartite graph~$G$ is {\it $\delta$-close} to a split colouring if by deleting at most $\delta |E(G)|$ edges we can make $G$ bipartite and the colouring a split colouring.

\begin{theorem}\label{thm:cycles}
For all $\delta>0$ there is a an $n_0$ such that the following holds for every  fair complete $k$-partite graph $G$ on $n>n_0$ vertices, with $k\geq 2$.\\ 
If the edges of $G$ are coloured red and blue, and the colouring is not $\delta$-close to a split colouring, then there are two disjoint monochromatic cycles of distinct colours,  which together cover all but at most $\delta n$ vertices of $G$.
\end{theorem}

It is easy to check that no colouring of a complete multipartite graph on $n$ vertices, with at least three partition classes of size greater than $2\sqrt{\delta} n$, can be $\delta$-close to a split colouring.
So, for these graphs we can drop the condition on the colouring in Theorem~\ref{thm:cycles}.
Also, notice that any complete multipartite graph on $n$ vertices, with a colouring that is $\delta$-close to a split colouring, contains  three disjoint monochromatic cycles that together cover all but at most $8\sqrt \delta n$ vertices of the graph.\footnote{In fact, in the graph induced by the edges of the split colouring, we can delete a balanced set of at most $2\sqrt \delta n$ vertices so that in the remaining graph $H$, each vertex has at most $\sqrt\delta n$ non-neighbours in the other partition class. Then we split $H$ into three monochromatic balanced bipartite graphs, two in blue, and one in red. It is easy to check that
each of these three graphs has a cycle covering all but at most $2 \sqrt \delta n$ of its vertices.}

We prove Theorem~\ref{thm:cycles} in Section~\ref{sec:cycles}. The strategy uses the regularity method, and a well-known technique due to \L uczak for blowing up connected matchings of the reduced graph to cycles in the original graph. The existence of the connected matchings in the given circumstances is shown in Lemma~\ref{lem:conn-matchings-robust} and Lemma~\ref{lem:conn-matchings-robust-for-bip} of Section~\ref{sec:matchings}.

Theorem~\ref{thm:cycles} is probably not the best possible result. It might be possible to cover all but a constant number of vertices of our multipartite graph. 
For an open problem in this direction, and more discussion, see Section~\ref{sec:conclu}.

\medskip

Using tools of Gy\'{a}rf\'{a}s, ~\cite{GRSS06b} and Haxell~\cite{Hax97}, we  derive from Theorem~\ref{thm:cycles} that a small finite number of monochromatic cycles is always sufficient to partition a multipartite graph.

\begin{theorem}\label{thm:covering-all}
Let $G$ be a sufficiently large fair complete $k$-partite graph, whose edges are coloured red and blue. 
Then $G$ can be partitioned into 14 monochromatic cycles. 
If $k=2$, then $G$ can be partitioned into 12 monochromatic cycles.
\end{theorem}

We prove Theorem~\ref{thm:covering-all} in Section~\ref{sec:exact-cycles}.
We believe that probably, the number of cycles can be dropped further, but the point of our result is rather that a reasonable finite number of cycles always suffices. 
\medskip

We end the introduction with a useful lemma, which tells us that for all our results for fair $k$-partite graphs with $k\geq 3$, we may restrict our attention to the tripartite case.

\begin{lemma}\label{lem:reductiontotripartite}
Every fair $k$-partite graph $G$ with $k \ge 3$ has a spanning induced subgraph which is fair and tripartite.
\end{lemma}
\begin{proof}
Assume $k \ge 4$ and delete all edges between the smallest two partition classes. If the resulting graph is not fair, then these two classes together contain more than $|V(G)|/2$ vertices. But then, also the third and the fourth smallest class together have more than $|V(G)|/2$ vertices, a contradiction. Inductively, the statement follows.
\end{proof}

\section{Proof of Theorem~\ref{thm:mainresult}}\label{sec:paths}

Throughout this section, let $G=(V,E)$ be a fair complete $k$-partite graph on $n$ vertices whose edges are coloured red and blue.
Let $V_1, \ldots , V_k$ be the partition classes of $G$ and let $n_i = |V_i|$ for $i = 1,\ldots,k$, with $n_1 \ge \ldots \ge n_k$.
Our aim is to partition $G$ into a red and a blue path.
By Lemma~\ref{lem:reductiontotripartite}, we may assume $k=3$.

\begin{lemma}\label{lem:equality-case}
If  $n_1=n/2$, then $G$ can be partitioned into a red and a blue path sharing a common endvertex.
\end{lemma}
\begin{proof}
Consider the balanced complete bipartite graph $H=(V,E_G(V_1,V_2 \cup V_3))$ whose partition classes are $V_1$ and $V_2 \cup V_3$.
By Theorem~\ref{thm:pokrovskiy}, $H$ and thus $G$ can be partitioned into a red and a blue path if the edges of $H$ are not split-coloured.
So, we may assume that the edges of $H$ are coloured with a split colouring, that is, there are disjoint non-empty sets $A,B \subseteq V_1$ and $C,D \subseteq V_2 \cup V_3$ with $A \cup B = V_1$ and $C \cup D = V_2 \cup V_3$, such that $E_G(A,C) \cup E_G(B,D)$ is entirely coloured blue, and $E_G(A,D) \cup E_G(B,C)$ is entirely coloured red.

Since $V_2, V_3\neq\emptyset$ we have that $E_G(C,D) \neq \emptyset$.
Let $uv \in E_G(C,D)$, say $uv$ is blue.
Let $P_u$ ($P_v$) be a blue path in $H$, of maximum even length starting in~$u$ (in $v$). 
Let $P$ be the blue path $P_u uv P_v$.
If $|A| \le |C|$, then $|D| \le |B|$, and thus $A \cup D \subseteq V(P)$.
Otherwise, $B \cup C \subseteq V(P)$.
In both cases $H[V \setminus V(P)]$ is a balanced complete bipartite graph with red edges only.
Hence, $G$ can be partitioned into a red and a blue path.

It is straightforward that there is an edge joining an endvertex of the red path to an endvertex of the blue path.
\end{proof}

The following is the main tool for our proof of Theorem~\ref{thm:mainresult}. 
\begin{lemma}\label{lem:veryfair}
Assume that $n_1 \le n_2 + n_3 - 2$, or that $n_1=n_2+n_3-1$ and $n_3>1$.
Then $G$ can be partitioned into a red and a blue path sharing a common endvertex.
\end{lemma}

The proof of Lemma~\ref{lem:veryfair} is the subject of Section~\ref{sec:uglylemmaproof}.

\begin{lemma}\label{lem:notsofair}
Assume that $n_1 = n_{2}$ and $ n_3 = 1$.
Then $G$ can be partitioned into a red and a blue path sharing a common endvertex.
\end{lemma}
\begin{proof}
Say $V_3 = \{z\}$.
Consider the balanced complete bipartite graph $H$ induced by the vertex set $V_1 \cup V_2$. First, assume that the edges of $H$ are not properly split-coloured. Then
by Theorem~\ref{thm:pokrovskiy}, $H$ can be partitioned into a red path $R$ and a blue path $B$.
Clearly, as $H$ is balanced, there are an endvertex $r$ of $R$ and an endvertex $b$ of $B$ which lie in distinct partition classes.
We may assume the edge $rb$ to be blue, the other case is analogous.
If the edge $rz$ is red, we may extend the red path to include $z$. Together with the edge $bz$, this gives the desired partition.
Otherwise, if $rz$ is blue, we may extend the blue path to include $z$. Together with the edge $r'z$, where $r'$ is the second to last vertex on $R$,  this gives the desired partition.

So we may assume that $H$ is properly split coloured. 
That is, there are disjoint non-empty sets $A,B \subseteq V_1$ and $C,D \subseteq V_2$ with $A \cup B = V_1$ and $C \cup D = V_2$, such that $E_G(A,C) \cup E_G(B,D)$ is entirely coloured blue, and $E_G(A,D) \cup E_G(B,C)$ is entirely coloured red. Now, there are two colour-components, either $A\cup D$ and $B\cup C$, or $A\cup C$ and $B\cup D$ which are connected in $G$ via $z$. We treat the case that there are vertices 
Let $a \in A$, $b \in B$ such that the edges $az$, $bz$ are red, all other cases can be treated analogously.

Choose a longest balanced red path $X$ starting in $a$, and a longest balanced red path $Y$ starting in $b$ (where balanced means the path should have an even number of vertices). Take a longest blue path $Z$ covering $G-(V(X)\cup V(Y) \cup \{z\})$.
The latter choice is possible since $A \cup C \subseteq V(X)\cup V(Y)$ or $B \cup D \subseteq V(X)\cup V(Y)$.
Thus, $G$ can be partitioned into the red path $XzY$ and the blue path~$Z$.
\end{proof}

\medskip

We are now ready for the proof of our first main theorem.

\begin{proof}[Proof of Theorem~\ref{thm:mainresult}]
By Lemma~\ref{lem:reductiontotripartite}, it suffices to prove our result for $k=3$. If either $n_1\leq n_2+n_3-2$, or $n_1=n_2+n_3-1$ and $n_3>1$, we may apply Lemma~\ref{lem:veryfair} and are done, so assume otherwise. Then either $n_1=n_2$ and $n_3=1$, in which case we may apply 
Lemma~\ref{lem:notsofair} and are done, or $n_1\geq n_2+n_3$, which we assume from now on. Since $G$ is fair, we actually have $n_1= n_2+n_3$, and are thus in conditions to apply Lemma~\ref{lem:equality-case} to obtain the desired partition.
\end{proof}

\medskip

We also make use of the following simple lemma we shall need for the proof of Lemma~\ref{lem:veryfair} in Section~\ref{sec:uglylemmaproof}.
Since the proof is straightforward, we omit it.

\begin{lemma}\label{lem:23edge-new}
Assume that $n_i \le n_{3-i} + n_3 - 1$ for some $i \in \{1,2\}$, and let $P$ be any Hamilton path in $G$.
Then either $P$ contains an edge $uv \in E_G(V_{3-i},V_3)$, or $n_i = n_{3-i} + n_3 - 1$ and both endvertices of $P$ are in $V_{3-i} \cup V_3$.
\end{lemma}


We end this section with the proof of Corollary~\ref{cor:path+cycle}.

\begin{proof}[Proof of Corollary~\ref{cor:path+cycle}.]
By Lemma~\ref{lem:reductiontotripartite}, we may assume $G$ is tripartite, with partition classes $V_1, V_2, V_3$.
We observe that the lemmas we use to prove Theorem~\ref{thm:mainresult}, Lemmas~\ref{lem:equality-case},~\ref{lem:veryfair}, and~\ref{lem:notsofair}, yield a partition of $G$ into a red and a blue path, say $R$ and $B$, that share their last vertex $x$.
Say among all such partitions, $R$ is chosen of maximum length. 
W.l.o.g.~assume $x\in V_1$.

If any of the two paths $R,B$ is trivial, or has only one edge, we are done. 
Note that we may assume all edges between the first two and  the last two vertices of $R$ to be blue, and all edges between the first two and the last two vertices of $B$ to be red, as otherwise we are done. 
In particular, by maximality of $R$, this implies that the first vertex $v_1$ on $B$ lies in $V_1$. 

First assume the first vertex $w_1$ on $R$  does not lie in $V_1$. 
Then $w_1x$ is blue, and by maximality of $R$, we know that $w_1v_1$ is blue, too. Thus we are done. 

So assume $w_1\in V_1$. 
By maximality of $R$, we know $w_1$ sends a blue edge to the second last vertex on $B$. 
Let $v_2$ be the second vertex on $B$.
Now if $v_2w_1$ is red we find a red cycle and a blue path covering all but one vertex, and if $v_2w_1$  is blue we find a blue cycle and a red path covering all but one vertex.
\end{proof}

\section{Proof of Lemma~\ref{lem:veryfair}}\label{sec:uglylemmaproof}

This section is devoted to the proof of Lemma~\ref{lem:veryfair}. For notational reasons it will be very convenient to now refrain from the assumption that $n_1\geq n_2$. We still keep the convention that $V_3$ is the smallest of the three classes, that is, $n_3\leq\min\{n_1,n_2\}$. We thus have to prove the following statement.\\

\noindent
{\it Assume that $n_i\geq n_3$ for $i\in\{1,2\}$, and that either $n_i \le n_{3-i} + n_3 - 2$, or $n_i=n_{3-i}+n_3-1$ and $n_3>1$.
Then $G$ can be partitioned into a red and a blue path sharing a common endvertex.}\\

The assumptions of the lemma imply that $\min\{n_1,n_2\} \ge 2$.
Let $v_1 \in V_1$ and $v_2 \in V_2$ be arbitrary, let $H=G-\{v_1,v_2\}$, and let $n'_1,n'_2,n'_3$ be the sizes of the partition classes of $H$.
If $n'_1,n'_2,n'_3$ satisfy the statement above, we may apply induction to see that $H$ can be partitioned into a red path $R$ and a blue path $B$ that share a common endvertex.

If $n_1',n_2',n_3'$ violate the statement above, it must be that $\min \{n_1',n_2'\} < n_3'$.
As $n_1'=n_1-1$ and $n_2'=n_2-1$, we may w.l.o.g.~assume that $n_2=n_3$.
So, $n_1 \ge n_2=n_3$.

Let us first discuss the case $n_1=n_2=n_3$.
It must be that $n_1 = 2$, for otherwise $n_3' \le n_1'+n_2'-1$, $n_1' \le n_2'+n_3'-1$ and $n_2'>1$, a contradiction to our assumption that $n_1',n_2',n_3'$ violate the statement.
Thus $n_3' = n_1' + n_2'$ and so we may apply Lemma~\ref{lem:equality-case} to obtain a red path $R$ and a blue path $B$ sharing a common endvertex that partition $H$.

Now assume $n_1 > n_2=n_3$.
Hence, $n_1' \ge n_3' > n_2' \ge 1$.
Again we have $n_1' \le n_2'+n_3'-1$, $n_3' \le n_1'+n_2'-1$.
Since $n_1',n_2',n_3'$ violate the statement, it must be that $n_2'=1$.
As $n_2 \ge n_3$ and $n_2' = n_2-1$, we have $n_2=2$ and $n_3'=n_3=2$.
Therefore $n_1 = 3$ and thus $n_1'=2$.
So, we can apply Lemma~\ref{lem:notsofair} to $H$ and obtain a red path $R$ and a blue path $B$ sharing a common endvertex that partition $H$.

Summing up, we may inductively assume Lemma~\ref{lem:veryfair} to hold for $H$.
Hence, $H$ can be partitioned into a red path $R$ and a blue path $B$, both possibly trivial, that share a common endvertex.
Let $R = (r_1,\ldots,r_s,x)$ and $B = (b_1,\ldots,b_t,x)$ where $R$ and $B$ have only $x$ in common.

Throughout the proof we suppose for contradiction that $G$ cannot be partitioned into a red and a blue path that share a common endvertex.
In several cases treated below we make use of the following simple fact.
Assume we can partition $G$ into a red path $R'$ and a blue path $B'$ such that one of these paths has its endvertices in distinct partition classes.
Then one of $R'$, $B'$ can be extended such that the paths have exactly one vertex in common, namely an endvertex of both paths.
The same holds if an endvertex of one path is in a distinct partition class than an endvertex of the other path.

For the remainder of the proof, we assume w.l.o.g.~that $v_1v_2$ is red.
At this point, we advise the reader to get his coloured pencils ready.

\begin{claim}\label{clm:non-trivial}
Neither $R$ nor $B$ is trivial. (That is $R\neq(x)\neq B$.)
\end{claim}
\begin{proof}
We first suppose that the path $R$ is trivial, that is, $B$ covers $H$.
W.l.o.g. assume~$b_1 \notin V_1$.
If the edge $b_1v_1$ is red, then $G$ can be partitioned into the red path $(v_2,v_1,b_1)$ and the blue path $(b_1,\ldots,b_t,x)$, a contradiction.
Otherwise, $G$ can be partitioned into the red path $(v_2,v_1)$ and the blue path $(v_1,b_1,\ldots,b_t,x)$, another contradiction.

Now suppose that $B$ is trivial, that is, $R$ covers $H$.
Again, we may w.l.o.g.~assume that $r_1 \notin V_1$.
The edge $r_1v_1$ is blue, since otherwise  the red path $(v_2,v_1,r_1,\ldots,r_s,x)$ covers $G$, a contradiction.

If $r_1 \in V_3$, then $r_1v_2 \in E$.
In this case, however, $G$ is covered by the red path $(v_1,v_2,r_1,\ldots,r_s,x)$, if $r_1v_2$ is red. 
If $r_1v_2$ is blue, $G$ can be partitioned into the red path $(r_2,\ldots,r_s,x)$ and the blue path $(v_1,r_1,v_2)$.
As both cases are contradictory, $r_1 \notin V_3$ and so $r_1 \in V_2$.
By symmetry, $x \notin V_3$.

First we suppose that $x \in V_1$.
Then $xv_2,xr_1 \in E$.
Like above, $xv_2$ must be blue.
But then $r_1x$ is red, for otherwise $G$ can be partitioned into the red path $(r_2,\ldots,r_{s})$ and the blue path $(v_1,r_1,x,v_2)$, a contradiction (because of the observation stated before the claim we are presently proving).
We may pick $i \in \{2,\ldots,s\}$ such that $r_i \in V_3$.
If the edge $v_1r_i$ is red, $G$ can be partitioned into the red path $(v_2,v_1,r_i,\ldots,r_s,x,r_1,\ldots,r_{i-1})$, a contradiction.
Thus, by symmetry, both $v_1r_i$ and $v_2r_i$ are blue.
This means $G$ can be partitioned into the red path $(r_{i+1},\ldots,r_s,x,r_1,\ldots,r_{i-1})$ and the blue path $(v_1,r_i,v_2)$, a contradiction.

So, $x \in V_2$.
Then $v_1x \in E$, and this edge must be blue, as it is interchangeable with $v_1r_1$.
Consider the edge $v_2r_2$: it must be blue, else $G$ can be partitioned into the red path $(v_1,v_2,r_2,\ldots,r_s,x)$ and the blue path $(v_1,r_1)$.
Thus, the edge $r_2x$ is red, since otherwise $G$ can be partitioned into the red path $(r_3,\ldots,r_{s-1}, r_s)$ and the blue path $(v_2,r_2,x,v_1,r_1)$.
But this is a contradiction: if the red path $(r_3,\ldots,r_{s})$ is non-empty, $r_{s} \notin V_2$ and so $v_2r_{s} \in E$.

By Lemma~\ref{lem:23edge-new}, we may pick $i \in \{2,\ldots,s\}$ such that $r_i \in V_3$ and $r_{i-1} \in V_1$ or $r_{i+1} \in V_1$.
Note that we may apply Lemma~\ref{lem:23edge-new} even if $n_i'=n_{3-i}'+n_3'-1$, for some $i \in \{1,2\}$, since $r_1 \in V_2$.
Let us assume that $r_{i-1} \in V_1$, the other case is similar.
If $r_iv_1$ is red, then $G$ can be partitioned into the red path $(v_2,v_1,r_i,\ldots,r_s,x,r_2,\ldots,r_{i-1})$ and the blue path $(r_1)$, a contradiction.
Similarly, $r_iv_2$ cannot be red.
So, both $r_iv_1$ and $r_iv_1$ are blue.
But now $G$ can be partitioned into the red path $(r_{i+1},\ldots,r_s,x,r_2,\ldots,r_{i-1})$ and the blue path $(v_2,r_i,v_1,r_1)$, which is contradictory.
This completes the proof of Claim~\ref{clm:non-trivial}.
\end{proof}

Over the next few claims, we deal with the case that $x \in V_3$.

\begin{claim}\label{xV3-a}
If $x \in V_3$, then $r_1\notin V_3$.
\end{claim}
\begin{proof}
Suppose $x, r_1 \in V_3$. Consider the edges $xv_1$ and $xv_2$. If any of these edges is red, then the respective  path $(r_1,\ldots,r_s,x,v_1,v_2)$ or $(r_1,\ldots,r_s,x,v_2,v_1)$ is red, and together with the blue path $(b_1,\ldots,b_t)$ covers $G$. Thus $xv_1$ and $xv_2$ are blue.

If either of  $r_1v_1$, $r_1v_2$ is  red, we may simply extend $R$ from $r_1$ to $v_1$ and $v_2$. So these two edges are blue, too. Observe that one of the blue paths  $(b_1,\ldots,b_t,x,v_1,r_1,v_2)$,  $(b_1,\ldots,b_t,x,v_2,r_1,v_1)$ has its endpoints in different partition classes. Together with  the red path $(r_2,\ldots,r_s)$, this blue path covers $G$, a contradiction.
\end{proof}

\begin{claim}\label{xV3-b}
Let $i\in\{1,2\}$. If $x \in V_3$, then not both $r_1$ and $r_s$ lie in $V_i$.
\end{claim}
\begin{proof}
Because of symmetry, it is enough to prove this claim for $i=1$. So for contradiction suppose that $x \in V_3$ and $r_1,r_s\in V_1$.

As above, we see that $xv_1$ and $r_1v_2$ are blue. 
Also, $r_sv_2$ is blue, since otherwise $G$ can be partitioned into the red path $(r_1,\ldots,r_s,v_2,v_1)$ and the blue path $(v_1,x,b_t,\ldots,b_1)$.

We claim that
\begin{equation}\label{b1inV1}
b_1 \in V_1.
\end{equation}
Suppose  $b_1 \notin V_1$. Then $xv_2$ must be blue, since otherwise $G$ can be partitioned into the red path $(r_1,\ldots,r_s,x,v_2,v_1)$ and the blue path $(b_1,\ldots,b_t)$.
As $v_1b_1 \in E$, this edge must be red, for otherwise $G$ can be partitioned into the red path $(r_1,\ldots,r_s)$ and the blue path $(v_2,x,b_t,\ldots,b_1,v_1)$.
Thus, $r_1b_1$ is blue, else $G$ can be partitioned into the red path $(r_s,\ldots,r_1,b_1,v_1,v_2)$ and the blue path $(v_2,x,b_t,\ldots,b_2)$.
Hence, $G$ can be partitioned into the red path $(r_2,\ldots,r_s)$ and the blue path $(v_1,x,b_t,\ldots,b_1,r_1,v_2)$, a contradiction.
This proves~\eqref{b1inV1}.

Hence, the edge $b_1v_2$ exists, and 
\begin{equation}\label{b1v2red}
\text{$b_1v_2$ is red,}
\end{equation}
 since otherwise $G$ can be partitioned into the red path $(r_1,\ldots,r_s)$ and the blue path $(v_1,x,b_t,\ldots,b_1,v_2)$.
If $t=1$, $xb_1$ is blue by definition.
If $t \ge 2$ and $xb_1$ is red, $G$ can be partitioned into the red path $(r_1,\ldots,r_s,x,b_1,v_2,v_1)$ and the blue path $(b_2,\ldots,b_t)$ (note that by~\eqref{b1inV1} we know that $b_2\notin V_1$).
Thus, 
\begin{equation}\label{trueblue}
\text{$xb_1$ is blue.}
\end{equation}

We now show that 
\begin{equation}\label{btinV1}
b_t \in V_1.
\end{equation}
Indeed, suppose otherwise. Then the edge $b_tr_1$ exists.
If $b_tr_1$ is blue, then $G$ can be partitioned into the red path $(r_2,\ldots,r_s)$ and the blue path $(v_1,x,b_1,\ldots,b_t,$ $r_1,v_2)$.
This is contradictory since $v_1 \in V_1$ and $v_2 \in V_2$.

So $b_tr_1$ is  red.
Since $xb_t\in E$, we have $b_t \in V_2$.
Thus, $v_1b_t \in E$, and this edge must be blue: otherwise $G$ can be partitioned into the red path $(x,r_s,\ldots,r_1,b_t,v_1,v_2)$ and the by~\eqref{trueblue} blue path $(x,b_1,\ldots,b_t)$.
Moreover, $xv_2$ must be blue, else $G$ can be partitioned into the red path $(r_1,\ldots,r_s,x,v_2,v_1)$ and the blue path $(v_1,b_t,\ldots,b_1,)$.
But this means $G$ can be partitioned into the red path $(r_1,\ldots,r_s)$ and the blue path $(v_2,x,v_1,b_t,\ldots,b_1)$, a contradiction.
This proves~\eqref{btinV1}.

By~\eqref{btinV1}, the edge $b_tv_2$ exists.
However, this edge cannot be blue, as then $G$ can be partitioned into the red path $(r_s,\ldots,r_1)$ and the blue path $(r_s,v_2,b_t,\ldots,b_1,$ $x,v_1)$.
So, $b_tv_2$ is red.

We next show that
\begin{equation}\label{truered}
\text{$xr_1$ is red.}
\end{equation}

For contradiction, suppose that $xr_1$ is blue.
Then, in particular, $s \ge 2$, since $xr_s$ is red.
The edge $r_2v_1$ must be blue, as otherwise $G$ can be partitioned into the red path $(v_2,v_1,r_2,\ldots,r_s)$ and the blue path $(v_2,r_1,x,b_t,\ldots,b_1)$.
Moreover, $b_tr_2$ is blue, as follows.
Suppose $b_tr_2$ is red.
If $t=1$, $G$ can be partitioned into the red path $(v_1,v_2,b_t,r_2,\ldots,r_s,x)$ and the blue path $(r_1,x)$.
Otherwise, $b_{t-1}$ exists and does not belong to $V_1$.
Thus, $G$ can be partitioned into the red path $(v_1,v_2,b_t,r_2,\ldots,r_s)$ and the blue path $(r_1,x,b_1,\ldots,b_{t-1})$.
As both is contradictory, we see that $b_tr_2$ is blue.
But now the red path $(r_3,\ldots,r_s)$ and the blue path $(r_s,v_2,r_1,x,v_1,r_2,b_t,\ldots,b_1)$ partition $G$, a contradiction. This proves~\eqref{truered}.

We now apply Lemma~\ref{lem:23edge-new} to $H$, 
to see that there is an edge $uv$ in $R$ or $B$ with $u \in V_3$ and $v \in V_2$.
Note that we may apply Lemma~\ref{lem:23edge-new} even if $n_i'=n_{3-i}'+n_3'-1$, for some $i \in \{1,2\}$, since $r_1 \in V_1$ by assumption.
Since $r_s,b_t \in V_1$, we know that $v \notin \{r_s,b_t\}$, and thus $u \neq x$.

We claim that 
\begin{equation}\label{uvareonB}
\text{$u,v \in V(B)$.}
\end{equation}
Indeed, suppose $u,v \in V(R)$.
We assume $v$ lies between $u$ and $x$ on $R$, say $u=r_i$ and $v=r_{i+1}$, the other case is similar.
It must be that $r_iv_2$ is blue, since otherwise $G$ can be partitioned into the red path $(v_1,v_2,r_i,r_{i-1},\ldots,r_1,x,r_s,\ldots,r_{i+1})$ and the blue path $(b_1,\ldots,b_t)$.
Similarly, the edges $r_iv_1$ and $r_{i+1}v_1$ are blue.

Since $r_{i+1} \in V_2$ and $b_1 \in V_1$, $r_{i+1}b_1 \in E$.
If $r_{i+1}b_1$ is red, then $G$ can be partitioned into the red path $(v_1,v_2,b_1,r_{i+1},\ldots,r_s,x,r_1,\ldots,r_i)$ and the blue path $(b_2,\ldots,b_t)$ (recall that $b_1v_2$ is red by~\eqref{b1v2red}).
However, if $r_{i+1}b_1$ is blue, then $G$ can be partitioned into the red path $(r_{i+2},\ldots,r_s,x,r_1,\ldots,r_{i-1})$ and the blue path $(v_2,r_i,v_1,r_{i+1},b_1,\ldots,b_t)$.
This proves~\eqref{uvareonB}.

Say $v$ lies between $u$ and $x$ on $B$, say $u=b_i$ and $v=b_{i-1}$, the other case is similar. We now show that 
\begin{equation}\label{r1biblue}
\text{$r_1b_i$ is blue.}
\end{equation}

Indeed, suppose that the edge $r_1b_i$ is red.
Thus $b_iv_1$ must be blue: otherwise, $G$ can be partitioned into the red path $(v_2,v_1,b_i,r_1,\ldots,r_s)$ and the blue path $(b_{i-1},\ldots,b_1,x,b_t,\ldots,b_{i+1})$. Similarly, $b_iv_2$ must be blue.
Hence $b_{i-1}v_1$ must be red, as otherwise $G$ can be partitioned into the red path $(r_1,\ldots,r_s)$ and the blue path $(v_2,b_i,v_1,b_{i-1},\ldots,b_1,x,b_t,\ldots,b_{i+1})$.
Supposing $b_{i-1}r_1$ is red, $G$ can be partitioned into the red path $(v_2,v_1,b_{i-1},r_1,\ldots,r_s)$ and the blue path $(b_i,\ldots,b_t,x,b_1,\ldots,b_{i-2})$.
Thus, the edge $b_{i-1}r_1$ is blue, and so $G$ can be partitioned into the red path $(r_2,\ldots,r_s)$ and the blue path $(v_1,b_i,v_2,r_1,b_{i-1},\ldots,b_1,$ $x,b_t,\ldots,b_{i+1})$.
This proves~\eqref{r1biblue}.

Now we see that $v_1b_{i-1}$ must be red:
otherwise, $G$ can be partitioned into the red path $(r_2,\ldots,r_s)$ and the blue path $(v_2,r_1,b_i,\ldots,b_t,x,b_1,\ldots,b_{i-1},v_1)$.
Also, the edge $r_1b_{i-1}$ is blue, as otherwise we can cover $G$ with the red path $(v_2,v_1,b_{i-1},r_1,\ldots,r_s)$ and the blue path $(b_i,\ldots,b_t,x,b_1,\ldots,b_{i-2})$. 
Therefore, $b_iv_1$ is red, as otherwise $G$ can be partitioned into the red path $(r_2,\ldots,r_s)$ together with the blue path $(v_2,r_1,b_{i-1},\ldots,b_1,x,b_t,\ldots,b_i,v_1)$.

This implies that $r_sb_{i-1}$ is red, else $G$ could be covered by the red path $(r_2,\ldots,r_{s-1})$ and the blue path $(v_1,x,b_t,\ldots,b_i,r_1,v_2,r_s,b_{i-1},\ldots,b_1)$. But then we can cover $G$ with the red path $(v_1, b_{i-1}, r_s,\ldots, r_2)$ and the blue path $(b_{i-2},\ldots, b_1, x, v_2, r_1, b_i,\ldots, b_t)$ (here we use~\eqref{trueblue} and~\eqref{r1biblue}), giving the final contradiction.
\end{proof}

\begin{claim}\label{xV3-c}
Let $i\in\{1,2\}$. If $x \in V_3$, then not both $r_1\in V_i$ and $r_s\in V_{3-i}$ hold.
\end{claim}
\begin{proof}
For symmetry, we only need to treat the case $i=1$. So suppose 
$r_1\in V_1$ and ${r_s \in V_2}$.
Then 
\begin{equation}\label{rsv1blue}
\text{$r_sv_1$ is blue,}
\end{equation}
$r_2v_1$ is blue, since otherwise $G$ can be partitioned into the red path $(r_1,\ldots,r_s,$ $v_1,v_2)$ and the blue path $(b_1,\ldots,b_t,x)$.
Thus, 
 \begin{equation}\label{r1rsred}
\text{$r_1r_s$ is red,}
\end{equation}
 since otherwise $G$ can be partitioned into the red path $(r_2,\ldots,r_{s-1})$ and the blue path $(v_2,r_1,r_s,v_1,x,b_t,\ldots,b_1)$.

Let us show that
 \begin{equation}\label{againb1inV1}
\text{$b_1 \in V_1$.}
\end{equation}
Suppose otherwise.
Then $r_1b_1 \in E$, and this edge must be red: if it was blue, $G$ could be covered by the red path $(r_2,\ldots,r_s)$ together with the blue path $(r_s,v_1,x,b_t,\ldots,b_1,r_1,v_2)$.
Moreover, $v_1b_1 \in E$, and this edge must be blue, since otherwise $G$ can be partitioned into the red path $(x,r_s,\ldots,r_1,b_1,v_1,v_2)$ and the blue path $(x,b_t,\ldots,b_2)$.
Also, $xv_2$ is blue, else $G$ can be partitioned into the red path $(r_1,\ldots,r_s,x,v_2,v_1)$ and the blue path $(b_1,\ldots,b_t)$.
Now, however, $G$ can be partitioned into the red path $(r_1,\ldots,r_s)$ and the blue path $(v_2,x,b_t,\ldots,b_1,v_1)$, a contradiction. This proves~\eqref{againb1inV1}.

By~\eqref{againb1inV1}, we know that
$v_2b_1 \in E$, and this edge must be red, as otherwise $G$ can be partitioned into the red path $(r_1,\ldots,r_s)$ and the blue path $(r_s,v_1,x,b_t,\ldots,b_1,v_2)$ (recall that $r_sv_1$ is blue by~\eqref{rsv1blue}).
Thus, $r_sb_1$ is blue, else $G$ can be partitioned into the red path $(r_1,\ldots,r_s,b_1,v_2,v_1)$ and the blue path $(v_1,x,b_t,\ldots,b_1)$.

Let $j \in \{1,2\}$ such that $n_j \ge n_{2-j}$.
Hence $n_j \le n_{2-j} + n_3 - 2$, so $n_3 \ge 2$.
Thus, there is a vertex $y \in V_3 \setminus \{x\}$. We show that
 \begin{equation}\label{yinB}
\text{$y \in V(B)$.}
\end{equation}
Suppose otherwise, say $y = r_i$.
Then, $r_iv_2$ is blue, as else $G$ can be partitioned into the red path $(v_1,v_2,r_i,\ldots,r_2,r_1,r_s,\ldots,r_{i+1})$ (recall that $r_1r_s$ is red by~\eqref{r1rsred}) and the blue path $(b_1,\ldots,b_t,x)$.
Similarly, $r_iv_1$ is blue.
But now $G$ can be partitioned into the red path $(r_{i-1},\ldots,r_1,r_s,\ldots,r_{i+1})$ and the blue path $(v_2,r_i,v_1,x,b_t,\ldots,b_1)$. This proves~\eqref{yinB}.

So say $y = b_i$.
Thus, $v_2b_i$ is red, otherwise $G$ can be partitioned into the red path $(r_1,\ldots,r_{s-1})$ and the blue path $(v_2,b_i,\ldots,b_t,x,v_1,r_s,b_1,\ldots,b_{i-1})$.
The edge $r_1b_i$ is red, too, else $G$ can be partitioned into the red path $(r_2,\ldots,r_{s-1})$ and the blue path $(v_2,r_1,b_i,\ldots,b_t,x,v_1,r_s,b_1,\ldots,b_{i-1})$. 
(For this, note that $r_{s-1}\notin V_2$.)
But now $G$ can be partitioned into the red path $(v_2,b_i,r_1,\ldots,r_{s-1})$ and the blue path $(b_{i+1},\ldots,b_t,x,v_1,r_s,b_1,\ldots,b_{i-1})$, a contradiction. This finishes the proof of the claim.
\end{proof}

Putting Claims~\ref{xV3-a}, ~\ref{xV3-b} and~\ref{xV3-c} together, and noting that if $x\in V_3$, then $r_s\notin V_3$, we obtain the following assertion.
\begin{claim}\label{xnotinV3}
$x\notin V_3$.
\end{claim}

We now turn to the case that $x\notin V_3$.
We first show an auxiliary claim.

\begin{claim}\label{xinV1r1nichtinV3}
Let $i\in\{1,2\}$. If $x\in V_i$ then the edges $xv_{3-i}$, $r_sv_i$ are blue, and $r_1\notin V_{3}$. Furthermore, if $r_1\in V_j$ for some $j\in\{1,2\}$, then $r_1v_{3-j}$ is blue.
\end{claim}

\begin{proof}
For symmetry, it is enough to show this claim for $i=1$. So, assume $x\in V_1$.
Then, the edge $xv_2$ is present in $G$.
If $xv_2$ is red, then the red path $(v_1,v_2,x,r_s,\ldots,r_1)$ and the blue path $(b_1,\ldots,b_t)$ together cover $G$. So we know that 
$xv_2$ is blue.

The edge $r_sv_1$ is present in $G$ as $r_s \notin V_1$.
If this edge is red, $G$ can be partitioned into the red path $(r_1,\ldots,r_s,v_1,v_2)$ and the blue path $(x,b_1,\ldots,b_t)$, a contradiction.
Thus, 
$r_sv_1$ is blue.

Similarly, neither of the edges $r_1v_1$ and $r_1v_2$, if present, can be red.
On the other hand, not both $r_1v_1$ and $r_1v_2$ can be present and blue, else $G$ can be partitioned into the red path $(r_2,\ldots,r_s)$ and the blue path $(v_1,r_1,v_2,x,b_t,\ldots,b_1)$.
Thus, either $r_1v_1$ or $r_1v_2$ is absent from $G$, which implies that $r_1 \notin V_3$. This proves the claim.
\end{proof}

\begin{claim}\label{xinV1r1inV2gehtnicht}
Let $i\in\{1,2\}$. If $x\in V_i$ then $r_1\notin V_{3-i}$.
\end{claim}

\begin{proof}
For symmetry, it is enough to show this claim for $i=1$. So, for contradiction, assume $x\in V_1$ and $r_1\in V_{2}$. 
By Claim~\ref{xinV1r1nichtinV3}, the edges $xv_{2}$, $r_sv_1$ and $r_1v_1$ are blue.

We show first that
\begin{equation}\label{b1nichtinV1}
 \text{$b_1 \notin V_1$.}
 \end{equation}
 Suppose otherwise.
Then clearly $r_1b_1$ is red, since otherwise $G$ can be partitioned into the red path $(r_2,\ldots,r_s)$ and the blue path $(v_1,r_1,b_1,\ldots,b_t,x,v_2)$.
Thus the edge $v_2b_1$ is blue, otherwise $G$ can be partitioned into the red path $(v_1,v_2,b_1,r_1,\ldots,r_s)$ and the blue path $(x,b_t,\ldots,b_2)$.
A symmetric argument shows that $r_1x$ is red.
Moreover, $r_sb_1$ is red, for otherwise $G$ can be partitioned into the red path $(r_2,\ldots,r_{s-1})$ and the blue path $(r_1,v_1,r_s,b_1,\ldots,b_t,x,v_2)$.
Also, the edge $b_tv_1$ is red, as otherwise $G$ can be partitioned into the red path $(r_2,\ldots,r_{s-1})$ and the blue path $(r_1,v_1,b_t,\ldots,b_1,v_2,x)$.
Hence, $b_1b_t$ is blue, else $G$ can be partitioned into the red path $(v_2,v_1,b_t,b_1,r_1,\ldots,r_s,x)$ and the blue path $(b_2,\ldots,b_{t-1})$.

Let $y \in V_3 \setminus \{x\}$.
First suppose that $y = r_i$ for some $1 \le i \le s$.
Thus both $v_1r_i$ and $v_2r_i$ are blue: if $v_1r_i$ is red, say, $G$ can be partitioned into the red path $(v_2,v_1,r_i,\ldots,r_1,x,r_s,\ldots,r_{i-1})$ and the blue path $(b_1,\ldots,b_t)$, a contradiction. (Note that $b_1\in V_1$ and $b_t\notin V_1$.)
But this means $G$ can be partitioned into the red path $(r_{i+1},\ldots,r_s,x,r_1,\ldots,r_{i-1})$ and the blue path $(v_1,r_i,v_2,b_1,\ldots,b_t)$.

Thus $y = b_i$ for some $1 \le i \le t$.
Hence, $v_1b_i$ is red, as otherwise $G$ can be partitioned into the red path $(r_2,\ldots,r_s)$ and the blue path $(r_1,v_1,b_i,\ldots,b_1,$ $v_2,x,b_t,\ldots,b_{i+1})$.
So, the edge $r_1b_i$ must be blue, for otherwise $G$ can be partitioned into the red path $(v_1,b_i,r_1,\ldots,r_s)$ and the blue path $(b_{i+1},\ldots,b_t,x,v_2,$ $b_1,\ldots,b_{i-1})$.
Consequently, $G$ can be partitioned into the red path $(r_2,\ldots,r_s)$ and the blue path $(v_1,r_1,b_i,b_{i+1},\ldots,b_t,x,v_2,b_1,\ldots,b_{i-1})$.
This finishes the proof of~\eqref{b1nichtinV1}.

\medskip

Next, we show that
\begin{equation}\label{b1nichtinV3}
 \text{$b_1 \notin V_3$.}
 \end{equation}
 Suppose otherwise. 
Then $r_1b_1 \in E$ and this edge must be red, since otherwise $G$ can be partitioned into the red path $(r_2,\ldots,r_s)$ and the blue path $(v_1,r_1,b_1,\ldots,b_t,x,v_2)$.
In the case that $v_1b_1$ is red, $G$ can be partitioned into the red path $(v_2,v_1,b_1,r_1,\ldots,r_s)$ and the blue path $(x,b_t,\ldots,b_1)$.
So $v_1b_1$ is blue, thus $G$ can be partitioned into the red path $(r_1,\ldots,r_s)$ and the blue path $(v_2,x,b_t,\ldots,b_1,v_1)$.
This finishes the proof of~\eqref{b1nichtinV3}.

\medskip

Putting~\eqref{b1nichtinV1} and ~\eqref{b1nichtinV3} together, we see that $$b_1\in V_2.$$
 Then $b_1v_1$ must be red, as otherwise $G$ can be partitioned into the red path $(r_1,\ldots,r_s)$ and the blue path $(v_2,x,b_t,\ldots,b_1,v_1)$.
Hence $xb_1$ is blue, since otherwise $G$ can be partitioned into the red path $(r_1,\ldots,r_s,x,b_1,v_1,v_2)$ and the blue path $(b_2,\ldots,b_t)$.
Moreover, $b_tv_1$ is red, otherwise $G$ can be partitioned into the red path $(r_1,\ldots,r_s)$ and the blue path $(v_2,x,b_1,\ldots,b_t,v_1)$.

We claim that
\begin{equation}\label{rsundbtinV2}
 \text{$r_s,b_t \in V_2$.}
 \end{equation}

For this, first assume the edge $r_sb_t$ does exist.
If $r_sb_t$ is red, $G$ can be partitioned into the red path $(r_1,\ldots,r_s,b_t,v_1,v_2)$ and the blue path $(x,b_1,\ldots,b_t)$.
If $r_sb_t$ is blue, however, $G$ can be partitioned into the red path $(r_1,\ldots,r_{s-1})$ and the blue path $(v_1,r_s,b_t,\ldots,b_1,x,v_2)$. This shows that $r_sb_t$ does not exist.
Since $x \in V_1$, we get that $r_s,b_t \in V_i$ for some $i \in \{2,3\}$.

Now, if $r_s,b_t \in V_3$, then $r_sv_2 \in E$.
If $r_sv_2$ is red, $G$ can be partitioned into the red path $(r_1,\ldots,r_s,v_2,v_1)$ and the blue path $(x,b_t,\ldots,b_1)$.
But if $r_sv_2$ is blue, $G$ can be partitioned into the red path $(r_1,\ldots,r_{s-1})$ and the blue path $(v_1,r_s,v_2,x,b_t,\ldots,b_1)$.
This proves~\eqref{rsundbtinV2}.

Since $r_1 \in V_2$, we have that $r_2b_t \in E$. We show that
\begin{equation}\label{r2btblau}
 \text{$r_2b_t$ is blue.}
 \end{equation}

Suppose otherwise.
Then also $r_1b_{t-1}$ is red, as else $G$ can be partitioned into the red path $(v_2,v_1,b_t,r_2,\ldots,r_s,x)$ and the blue path $(b_1,\ldots,b_{t-1},r_1)$.
Hence, the edge $v_2b_{t-1}$ is blue, for otherwise $G$ can be partitioned into the red path $(b_t,v_1,v_2,b_{t-1},r_1,\ldots,r_s,x)$ and the blue path $(b_1,\ldots,b_{t-2})$.
So, $r_sb_{t-1}$ is blue, else $G$ can be partitioned into the red path $(v_2,v_1,b_t,r_2,\ldots,r_s,b_{t-1},r_1)$ and the blue path $(x,b_1,\ldots,b_{t-2})$.
Now, however, $G$ can be partitioned into the red path $(b_t,r_2,\ldots,r_{s-1})$ and the blue path $(r_1,v_1,r_s,b_{t-1},v_2,x,b_1,\ldots,b_{t-2})$. This proves~\eqref{r2btblau}.

Observe that 
\begin{equation}\label{r2rsrot}
 \text{$r_2r_s$ is red,}
 \end{equation}
 for otherwise $G$ can be partitioned into the red path $(r_{s-1},\ldots,r_3)$ and the blue path $(r_1,v_1,r_s,r_2,b_t,\ldots,b_1,x,v_2)$. 

We now apply Lemma~\ref{lem:23edge-new} to $H$,
 to see that there is an edge $uv$ on $R$ or on $B$ with $u \in V_3$ and $v \in V_1$.
We may apply Lemma~\ref{lem:23edge-new} even if $n_i'=n_{3-i}'+n_3'-1$, for some $i \in \{1,2\}$, since $r_1 \in V_2$.

We claim that
\begin{equation}\label{uaufB}
 \text{$u \in V(B)$.}
 \end{equation}
Suppose otherwise, i.e.~assume $u = r_i$ for some $2 \le i \le t$.
We discuss the case $v = r_{i+1}$ only, the other case is similar.
Observe that not both $v_1r_i$ and $v_2r_i$ can be blue, for otherwise  the red path $(r_{i-1},\ldots,r_2,r_s,\ldots,r_{i+1})$ and the blue path $(r_1,v_1,r_i,v_2,x,b_t,\ldots,b_1)$ cover $G$.
Say $v_1r_i$ is red,  the other case can be resolved similarly.
If $r_1x$ is blue, then we can cover $G$ with the red path $(v_2,v_1,r_i,\ldots,r_1,r_s,r_{s-1},\ldots, r_{i+1})$ and the blue path $(r_1,x,b_t,\ldots,b_1)$.
Hence, $r_1x$ is red, and so $G$ can be partitioned into the red path $(v_2v_1,r_i,r_{i-1},\ldots,r_1,x,r_s,\ldots,r_{i+1})$ and the blue path $(b_1,\ldots,b_t)$.
This proves~\eqref{uaufB}.

Say $u = b_i$ for some $2 \le i \le t$.
We assume $v = b_{i-1}$, as the other case is similar.
If $v_1b_i$ is blue, $v_2r_2$ must be red: otherwise, $G$ can be partitioned into the red path $(r_3,\ldots,r_s)$ and the blue path $(r_1,v_1,b_i,\ldots,b_t,r_2,v_2,x,b_1,\ldots,b_{i-1})$.
But then $G$ can be partitioned into the red path $(v_2,r_2,\ldots,r_s)$ and the blue path $(r_1,v_1,b_i,\ldots,b_t,r_2,v_2,x,b_1,\ldots,b_{i-1})$.

Thus, $v_1b_i$ must be red.
Hence, $r_1b_i$ is blue, for otherwise we can cover  $G$ with the red path $(v_2,v_1,b_i,r_1,\ldots,r_s)$ and the blue path $(b_{i+1},\ldots,b_t,b_1,\ldots,b_{i-1})$.
So, the edge $r_2v_2$ is blue, as otherwise $G$ can be partitioned into the red path $(v_1,v_2,r_2,\ldots,r_s)$ and the blue path $(r_1,b_i,\ldots,b_t,x,b_1,\ldots,b_{i-1})$.
But then $G$ can be partitioned into the red path $(r_3,\ldots,r_{s-1})$ and the blue path $(r_s,v_1,r_1,b_i,\ldots,b_t,r_2,v_2,x,b_1,\ldots,b_{i-1})$, yielding the final contradiction.
\end{proof}

\begin{claim}\label{xinV1r1inV1gehtnicht}
Let $i\in\{1,2\}$. If $x\in V_i$ then $r_1\notin V_{i}$.
\end{claim}

\begin{proof}
Because of symmetry, we only show the claim for $i=1$. So suppose $x\in V_1$ and  $r_1\in V_{1}$.
By Claim~\ref{xinV1r1nichtinV3}, the edges $r_sv_1$, $v_2x$ and  $r_1v_2$ are blue.

We first show that
\begin{equation}\label{b1noestaenV1}
 \text{$b_1 \notin V_1$.}
 \end{equation}

Suppose otherwise.
Then the edge $r_sb_1$ is red, for otherwise $G$ can be partitioned into the red path $(r_2,\ldots,r_{s-1})$ and the blue path $(v_1,r_s,b_1,\ldots,b_t,x,v_2,$ $r_1)$.
Thus, the edge $b_1v_2$ is blue, else $G$ can be partitioned into the red path $(v_1,v_2,$ $b_1,r_s,\ldots,\ldots,r_1)$ and the blue path $(b_2,\ldots,b_t,x)$.
Moreover, $r_2v_1$ is blue, since otherwise $G$ can be partitioned into the red path $(v_1,r_2,\ldots,r_s)$ and the blue path $(r_1,v_2,b_1,\ldots,b_t,x)$.
Hence, $r_2b_1$ must be red, as otherwise $G$ can be partitioned into the red path $(r_{s-1},\ldots,r_3)$ and the blue path $(r_s,v_1,r_2,b_1,\ldots,b_t,x,v_2,r_1)$.
Thus, $v_1b_2$ must be red, as otherwise $G$ can be partitioned into the red path $(r_{s-1},\ldots,r_2,b_1)$ and the blue path $(r_s,v_1,b_2,\ldots,b_t,$ $x,v_2,r_1)$.

Applying Lemma~\ref{lem:23edge-new} to $H$, we see that there is an edge $uv$ in $R$ or $B$ with $u \in V_3$ and $v \in V_2$.
Note that we may apply Lemma~\ref{lem:23edge-new} even if $n_i'=n_{3-i}'+n_3'-1$, for some $i \in \{1,2\}$, since $r_1 \in V_1$ by assumption.

First assume ${u \in V(R)}$, i.e.~$u = r_i$ for some $2 \le i \le t$.
Say $v = r_{i-1}$, the other case is similar.
If $r_iv_1$ is red, then $G$ can be partitioned into the red path $(v_1,r_i,\ldots,r_2,b_1,r_s,\ldots,r_{i+1})$ and the blue path $(r_1,v_2,x,b_t,\ldots,b_2)$.
Thus $r_iv_1$ is blue and, similarly, $r_{i-1}v_1$ is blue.
Either $r_i$ or $r_{i-1}$ is adjacent to $b_2$, say $r_ib_2 \in E$ (the other case is analogous).
If $r_ib_2$ is red, $G$ can be partitioned into the red path $(v_1,b_2,r_i,\ldots,r_s,b_1,r_2,\ldots,r_{i-1})$ and the blue path $(r_1,v_2,x,b_t,\ldots,b_3)$.
Otherwise, $G$ can be partitioned into the red path $(r_{i-1},\ldots,r_2,b_1,r_s,\ldots,r_{i+1})$ and the blue path $(r_1,v_2,x,b_t,\ldots,b_2,r_i,v_1)$.
However, both is contradictory, and we may thus assume that $u \in V(B)$.

So, $u = b_i$ for some $2 \le i \le t$.
We discuss the case $v = b_{i-1}$ only, as the case $v = b_{i+1}$ can be treated similarly.
Now, the edge $b_iv_1$ is red, since otherwise $G$ can be partitioned into the red path $(r_1,\ldots,r_s)$ and the blue path $(v_1,b_i,b_{i+1},\ldots,b_t,x,v_2,b_1,\ldots,b_{i-1})$.
Similarly, $b_{i-1}v_1$ is red.
Clearly $r_1b_i \in E$, and this edge must be blue.
Otherwise, $G$ can be partitioned into the red path $(v_1,b_i,r_1,\ldots,r_s)$ and the blue path $(b_{i-1},\ldots,b_1,v_2,x,b_t,\ldots,b_{i+1})$.
Similarly, $r_1b_{i-1}$ is blue.

It is clear that $r_2b_i \in E$ or $r_2b_{i-1} \in E$, say $r_2b_i \in E$.
If the edge $r_2b_i$ is blue, then we can cover $G$ with the red path $(r_3,\ldots,r_s)$ and the blue path $(v_1,r_2,b_i,r_1,b_{i-1},\ldots,b_1,v_2,x,b_t,\ldots,b_{i+1})$.
Otherwise, we can partition $G$ into the red path $(v_1,b_i,r_2,\ldots,r_s)$ and the blue path $(r_1,b_{i-1},\ldots,b_1,v_2,x,b_t,\ldots,b_{i+1})$. This finishes the proof of~\eqref{b1noestaenV1}.

\medskip

Next, we show that
\begin{equation}\label{b1noestaenV3}
 \text{$b_1 \notin V_3$.}
 \end{equation}
 Suppose otherwise. 
Then $v_1b_1 \in E$ and this edge must be red, otherwise $G$ can be partitioned into the red path $(r_1,\ldots,r_s)$ and the blue path $(v_1,b_1,\ldots,b_t,$ $x,v_2)$.
 
Now assume  $r_s\in V_2$. In that case, if $r_sb_1$ is red, we cover $G$ with the red path $(v_2,v_1,b_1,r_s,\ldots, r_1)$ and the blue path $(b_2,\ldots, b_t,x)$. If  $r_sb_1$ is blue, we cover $G$ with the red path $(r_{s-1},\ldots, r_1)$ and the blue path $(v_1,r_s,b_1,b_2,\ldots, b_t,x,v_2)$. 
Thus $r_s\notin V_2$ and hence $r_s\in V_3$.

But now, the edge $r_sv_2$ cannot be red, because of the red path $(r_1,\ldots, r_s,v_2,$ $v_1)$ and the blue path $(b_1\ldots, b_t,x)$. It also cannot be blue, because of the red path $(r_1,\ldots, r_{s-1})$ and the blue path $(b_1\ldots, b_t,x,v_2,r_s,v_1)$.
This finishes the proof of~\eqref{b1noestaenV3}.

\medskip

By~\eqref{b1noestaenV1} and~\eqref{b1noestaenV3}, we know that ${b_1 \in V_2}$.
Thus $b_1v_1 \in E$ and this edge must be red, as otherwise $G$ can be partitioned into the red path $(r_1,\ldots,r_s)$ and the blue path $(v_1,b_1,\ldots,b_t,x,v_2)$.
So, $b_1r_1$ is blue, for otherwise $G$ can be partitioned into the red path $(v_2,v_1,b_1,r_1,\ldots,r_s)$ and the blue path $(x,b_t,\ldots,b_1)$.
Thus, $b_tv_1$ is red, since otherwise $G$ can be partitioned into the red path $(r_2,\ldots,r_s)$ and the blue path $(v_1,b_t,x,v_2,r_1,b_1,\ldots,b_{t-1})$.
This implies that the edge $b_tr_1$ is blue, as otherwise $G$ can be partitioned into the red path $(v_2,v_1,b_t,r_1,\ldots,r_s,x)$ and the blue path $(b_{t-1},\ldots,b_1)$.
Moreover, $r_1r_s$ is red, for otherwise $G$ can be partitioned into the red path $(r_2,\ldots,r_{s-1})$ and the blue path $(v_1,r_s,r_1,v_2,x,b_t,\ldots,b_1)$.
Also, $xb_1$ is blue, else $G$ can be partitioned into the red path $(v_2,v_1,b_1,x,r_s,\ldots,r_1)$ and the blue path $(b_{t-1},\ldots,b_2)$.

Applying Lemma~\ref{lem:23edge-new} to $H$, we see that there is an edge $uv$ in $R$ or $B$ with $u \in V_3$ and $v \in V_1$.
Note that we may apply Lemma~\ref{lem:23edge-new} even if $n_i'=n_{3-i}'+n_3'-1$, for some $i \in \{1,2\}$, since we know that $b_1 \in V_2$.

First assume ${u \in V(R)}$, i.e.~$u = r_i$ for some $2 \le i \le t$.
We may suppose that $v = r_{i-1}$, the other case is similar.
Then $r_iv_1$ is blue, for otherwise $G$ can be partitioned into the red path $(v_2,v_1,r_i,\ldots,r_s,r_1,\ldots,r_{i-1})$ and the blue path $(x,b_t,\ldots,b_1)$.
Similarly, $r_iv_2$ is blue.
But then $G$ can be partitioned into the red path $(r_{i-1},\ldots,r_1,r_s,\ldots,r_{i+1})$ and the blue path $(v_1,r_i,v_2,x,b_t,\ldots,b_1)$, a contradiction.

Therefore,  ${u \in V(B)}$.
Say $u = b_i$ for some $2 \le i \le t$.
Say $v = b_{i-1}$, the other case is similar.
Thus, $b_iv_1$ is red, for otherwise $G$ can be partitioned into the red path $(r_2,\ldots,r_s)$ and the blue path $(v_1,b_i,\ldots,b_1,r_1,v_2,x,b_t,\ldots,b_{i+1})$.
Hence, the edge $r_1b_i$ is blue, for otherwise $G$ can be partitioned into the red path $(v_2,v_1,b_i,r_1,\ldots,r_s)$ and the blue path $(b_{i-1},\ldots,b_1,x,b_t,\ldots,b_{i+1})$.

If $r_s \in V_3$, then $r_sv_2 \in E$.
But then $r_sv_2$ must be blue, since otherwise $G$ can be partitioned into the red path $(r_1,\ldots,r_s,v_2,v_1)$ and the blue path $(x,b_t,\ldots,b_1)$.
This implies that $G$ can be partitioned into the red path $(r_1,\ldots,r_{s-1})$ and the blue path $(v_1,r_s,v_2,x,b_t,\ldots,b_1)$.
So, $r_s \notin V_3$, and hence $r_s \in V_2$.
Thus, the edge $r_sb_i$ exists and it must be red, for otherwise $G$ can be partitioned into the red path $(r_2,\ldots,r_{s-1})$ and the blue path $(v_1,r_s,b_i,\ldots,b_t,x,v_2,r_1,b_1,\ldots,b_{i-1})$, where $v_1 \in V_1$ and $r_2 \notin V_1$.
Hence, $G$ can be partitioned into the red path $(v_2,v_1,b_i,r_s,\ldots,r_1)$ and the blue path $(b_{i+1},\ldots,b_t,x,b_1,\ldots,b_{i-1})$, a contradiction.
\end{proof}

Putting Claims~\ref{xnotinV3}, \ref{xinV1r1inV2gehtnicht} and~\ref{xinV1r1inV1gehtnicht} together, we arrive at the final contradiction and thus 
complete the proof of Lemma~\ref{lem:veryfair}.

\section{Connected matchings}\label{sec:matchings}

Let $G$ be a graph whose edges are coloured red or blue.
A monochromatic matching of $G$ is called \emph{connected} if it is contained in a connected component of the subgraph induced by the edges of the corresponding colour. 

The following lemma, Lemma~\ref{lem:conn-matchings-exact}, says that any fair complete multipartite graph with at least three partition classes can be covered with two connected matchings of distinct colours. This is a direct consequence of our Theorem~\ref{thm:mainresult}, and thus there would be no reason to prove such a statement. But, as our aim is to apply the regularity method later, in order to pump up our paths/connected matchings to cycles that cover almost all vertices of the graph, we need a robust version of Lemma~\ref{lem:conn-matchings-exact}. (This version is given in Lemma~\ref{lem:conn-matchings-robust} below.) It will be much easier to extend the proof of Lemma~\ref{lem:conn-matchings-exact} for a robust version, than the one of Theorem~\ref{thm:mainresult}. We prefer to give the proof of the exact version first, so that the idea becomes clear to the reader. 

It will be convenient to formulate the two lemmas only for tripartite graphs. This is justified by Lemma~\ref{lem:reductiontotripartite}.

\begin{lemma}\label{lem:conn-matchings-exact}
Let $G$ be a fair complete tripartite graph on an even number of vertices.
If the edges of $G$ are coloured red or blue, then there are two vertex-disjoint connected matchings of distinct colours that together cover all vertices of $G$.
\end{lemma}
\begin{proof}
Say $V_1$ and $V_2$ are the largest two partition classes of $G$.
Let $v_1 \in V_1$, $v_2 \in V_2$ and set $G' = G -\{v_1,v_2\}$.
Note that $G'$ is fair, and unless $|V_2|=1$, we may apply induction to obtain two vertex-disjoint connected matchings of distinct colours that together cover all vertices of $G'$. 

On the other hand, if $|V_2|=1$, then $|V_3| = 1$, and thus by fairness, $|V_1|=2$.
In this case, any two disjoint edges $e_1,e_2$ cover all the vertices of $G'$. Clearly, it is either possible to choose $e_1$, $e_2$ of distinct colours, or in way that they give a monochromatic connected matching.

Resuming, in either case there are a red and a blue connected matching $M_R$, $M_B$ that cover all of~$V(G)$ except $v_1,v_2$. Let $V_R$ be
 the vertex set of the connected red component of $G$ containing $M_R$, and let $V_B$ be the analogue for blue.
Even if one of the matchings is empty, note that we can always assume that  $|V_R|, |V_B|\geq 1$.

Also, $V_R$ and $V_B$ meet. Indeed,
choose any distinct $i,j \in \{1,2,3\}$ with $V_R \cap V_i\neq \emptyset\neq V_B \cap V_j $, say $x \in V_R \cap V_i$ and $y \in V_B \cap V_j$.
The edge $xy$ is either red or blue, which means that $(V_R \cap V_B)\cap (V_i\cup V_j) \neq \emptyset$. In particular, we get $V_R \cap V_B \neq \emptyset$.

Assume $v_1v_2$ is red.
Suppose that $G$ cannot be covered with two connected matchings as desired.
Thus, 
\begin{equation}\label{nichtinrot}
\{v_1,v_2\} \cap V_R = \emptyset,
\end{equation}
 since otherwise we could add $v_1v_2$ to $M_R$.
This means that 
\begin{equation}\label{blaublaublau}
\text{all edges in the cut $E_G(V_R,\{v_1,v_2\})$ are blue.}
\end{equation}
In particular, all edges in $E_G(V_R \cap V_B,\{v_1,v_2\})$ are blue.
Hence, at least one of $v_1,v_2$ is contained in $V_B$, say $v_1\in V_B$. 
Consequently, $V(G) \setminus \{v_2\} \subseteq V_R \cup V_B$.
Thus, by~\eqref{blaublaublau}, we have that 
\begin{equation}\label{allesblau}
V(G)\setminus (V_1\cup\{v_2\})\subseteq V_B.
\end{equation}

Now, if $v_2\notin V_B$, then by~\eqref{nichtinrot}, we have $v_2\notin V_R \cup V_B$, and hence $V_R \cap V_B\subseteq V_2$ (since any edge from $v_2$ to $V_R \cap V_B$ has some colour). 
By the argument we used for showing that $V_R \cap V_B\neq\emptyset$, we know that $V_1\cup V_3$ is contained either in $V_R\setminus V_B$ or in $V_B\setminus V_R$. 
By~\eqref{allesblau}, this means we have $V_B\setminus V_R\supseteq V_1\cup V_3$.

In either case, we find that  $V_B$ covers all but at most one partition class. 
Let $M_B'$ be a largest blue matching in $G$ such that $G'':=G-V(M'_B)$ is still fair.
Since $V_B$ covers all but at most one partition class, $M_B'$ is a connected matching.
If all edges in $G''$ are red, then $G''$ has a red connected perfect matching, and we are done. 

So assume there is some blue edge $uv \in E(G'')$. 
By the choice of $M'_B$, the graph $G''-\{u,v\}$ is not fair. This means neither of $u,v$ lies in the largest partition class $V''_1$ of $G''$, and, furthermore, $|V''_1|=|V(G'')|/2$. Thus all edges between $V''_1$ and the rest of $G''$ are red, and hence, we can cover $G''$ with a red connected matching.
 \end{proof}
 
We now give a robust version of Lemma~\ref{lem:conn-matchings-exact}. This is the result we need for applying regularity later. 
 
\begin{lemma}\label{lem:conn-matchings-robust}
Let $G$ be a fair tripartite graph on $n$ vertices, with partition classes $V_1,V_2,V_3$, such that $|V_i|\geq 3\eps n$, for some~$\eps$ with $0<\varepsilon <1/5$. Suppose that $$d(v)\geq (1-\varepsilon) (n-|V_i|)$$ for $i=1,2,3$ and for each $v\in V_i$. If  the edges of $G$ are coloured red and blue, then there are two vertex-disjoint connected matchings of distinct colours that together cover all but at most $36\varepsilon n$ vertices of $G$.
\end{lemma}

\begin{proof}
Take a red connected matching $M_R$ and a blue connected matching $M_B$, which together cover as much as possible of  $V(G)$, while leaving $G':=G-V(M_R\cup M_B)$ fair. Let $V_R$, $V_B$ be the vertex sets of the respective colour components, as above we may assume both are non-empty. We assume $G'$ contains more than $36\eps n$ vertices. 

We claim that for all $i,j\in\{1,2,3\}$ with $i\neq j$ and $V_i\cap V_R\neq\emptyset$ and $|V_j\cap V_B| > \eps n$, we have
\begin{equation}\label{inter}
(V_R\cap V_B)\cap (V_i\cup V_j)\neq \emptyset.
\end{equation}
Indeed, by the degree condition of the lemma, there is an edge from $V_i\cap V_R$ to $V_j\cap V_B$, and clearly, one of its endpoints lies in $V_R\cap V_B$. This proves~\eqref{inter}. Observe that~\eqref{inter} also holds if we interchange the roles of $V_R$ and $V_B$.

Let $$V_\eps:=V(G)\setminus (V_R\cup V_B).$$
Our next aim is to show that 
there are $\ell\in\{1,2,3\}$, $X\in\{R,B\}$ with 
\begin{enumerate}[(a)]
\item $V_X\supseteq V(G)\setminus (V_\ell\cup V_\eps)$, and\label{onelargecolour}
\item $|V_\eps\cap V_i|\leq\eps n$ for $i\in\{1,2,3\}$, $i\neq\ell$.\label{otherssmall}
\end{enumerate}

First of all, since all edges are coloured,
note that no vertex in $V_\eps$ may send an edge to $V_R\cap V_B$. So, by the degree condition of the lemma, we know that 
$|V_\eps\cap V_\ell|\leq \eps n$ for all $\ell$ with $(V_R\cap V_B)\setminus V_\ell\neq\emptyset$.
In particular, if $V_\ell$ is such that $|V_\eps\cap V_\ell|> \eps n$, then $V_R \cap V_B \subseteq V_\ell$.
This implies $|V_\eps\cap V_i|\leq \eps n$ and $|V_\eps\cap V_j|\leq \eps n$, where $V_i, V_j$ are the other two partition classes. Moreover, since $V_j$ has at least $3\eps n$ vertices, we know that $V_j\setminus V_\eps$ has at least $\eps n$ vertices in either $V_R\setminus V_B$ or in $V_B\setminus V_R$, say in $V_B\setminus V_R$. Then by~\eqref{inter}, we have that $ V_i\setminus V_\eps$ is contained in $V_B\setminus V_R$. Again by~\eqref{inter} (interchanging the roles of $i$ and $j$), we see that also $ V_j\setminus V_\eps$ is contained in $V_B\setminus V_R$.  Thus, either~\eqref{onelargecolour} and~\eqref{otherssmall} hold, or
\begin{equation}\label{allsmallor}
\text{$|V_\eps\cap V_\ell|\leq \eps n$ for all $\ell\in\{1,2,3\}$.}
\end{equation}
We now assume the latter assertion in order to show that~\eqref{onelargecolour} and~\eqref{otherssmall} hold also in this case.

Since $|V(G')|>36\eps n$, and since $G'$ is fair, the largest two partition classes $V'_1$, $V'_2$ of $G'$  each contain at least $9\eps n$ vertices. For notational ease, we assume $V'_i\subseteq V_i$ for $i=1,2$.

A well-known fact states that any graph $H$ has a subgraph $H'$ which has minimum degree at least half the average degree of $H$. Applying this fact to the bipartite graph spanned by $V'_1$ and $V_2'$, in the majority colour, say this is red, we obtain sets $U_1\subseteq V'_1$ and $U_2\subseteq V'_2$ such that the minimum degree from $U_i$ to $U_{3-i}$ in red  is greater than $2\eps n$. In particular, $|U_i|> 2\eps n$, for $i=1,2$.

By maximality of $M_R$, and since $G'-\{u_1,u_2\}$ is still fair for any red edge $u_1u_2$ between $U_1$ and $U_2$, we know that all edges between $U_1\cup U_2$ and~$V_R$ are blue. In particular, the edges between $U_1\cup U_2$ and some fixed $x\in V_R\cap V_B$ are blue. Without loss of generality, assume $x\notin V_1$. Since $U_1$ has size greater than $2\varepsilon n$, every vertex $y\in V_R\setminus V_1$ sends a blue edge to some blue neighbour of $x$ in $U_1$. Thus, by the definition of $V_\eps$, we have
\begin{equation*}
V(G)\setminus (V_\eps\cup V_1)\subseteq V_B,
\end{equation*}
which is as desired for~\eqref{onelargecolour}. Because of~\eqref{allsmallor}, also~\eqref{otherssmall} holds. We have thus shown~\eqref{onelargecolour} and~\eqref{otherssmall}; let us assume they hold for $\ell =1$ and $X=B$.

\medskip

Set $V'_\eps:=V_\eps\setminus V_1$.
Now take a maximal blue matching $M'_B$ in $G-V'_\eps$ such that $G'':=G-V(M'_B)$ is still fair; by~\eqref{onelargecolour} we know that $M'_B$ is connected in blue. Assume that $|V(G'')| >36\eps n$, as otherwise we are done. 

By maximality of $M'_B$, for any blue edge in $E(G''-V'_\eps)$ we have that $G''-e$ is not fair. Thus, similar as in the proof of Lemma~\ref{lem:conn-matchings-exact},  either all edges of $G''-V'_\eps$ are red, or $G''$ contains a  spanning balanced bipartite graph $H$ such that any blue edge in $H$ is incident with $V'_\eps$. 

In either case, we can easily find a red connected matching in $G''$ covering almost all of $G''$ as follows. Take a maximal red connected matching $M'_R$ in $G''-V'_\eps$, or in  $H-V'_\eps$, such that the remainder of the graph is still fair. Note that we may assume $M'_R\neq\emptyset$ as $|V'_\eps|\leq 2\eps n$ by~\eqref{otherssmall}.
 Let $xy\in M'_R$. The neighbourhoods of $x$ and $y$ in the uncovered part of $G''-V'_\eps$ or $H-V'_\eps$ are large enough to span at least one red edge. Also, all edges between $x$ and $G''-V'_\eps$ or $H-V'_\eps$ are red, a contradiction to the maximality of $M'_R$.
\end{proof}

We now give an analogue of Lemma~\ref{lem:conn-matchings-robust} for bipartite graphs. With the obvious exclusion of the proper split colouring, we can cover all $2$-edge-coloured balanced bipartite graphs that are almost complete bipartite, with two connected matchings of distinct colours.

\begin{lemma}\label{lem:conn-matchings-robust-for-bip}
Let $G$ be a balanced bipartite graph on $n$ vertices, with partition classes $V_1,V_2$.  Suppose $G$ has minimum degree at least $(1-\eps)n/2$,  for some~$\eps$ with $0<\varepsilon <1/5$.\\ If  the
edges of $G$ are coloured red and blue, and this colouring is not a split colouring,
then there are two vertex-disjoint connected matchings of distinct colours that together cover all but at most $4\varepsilon n$ vertices of $G$.
\end{lemma}

\begin{proof}
Take a red connected matching $M_R$ and a blue connected matching $M_B$, together covering as much as possible of  $V(G)$. Let $V_R$, $V_B$ be the vertex sets of the respective colour components. We assume $G-V(M_R\cup M_B)$ contains more than $4\eps n$ vertices.  As above we see that
$V_R\cap V_B\neq \emptyset$, say there is a vertex $x\in V_R\cap V_B\cap V_2$.

Now, if there is an $X\in\{R,B\}$ and $i\in\{1,2\}$ such that $V_X$ covers all but a set $V'_i$ of at most $2\eps n$ vertices of $V_i$, we can proceed as in the proof of Lemma~\ref{lem:conn-matchings-robust}. That is, we take a maximal matching $M'_X$ in $G-V'_i$ in the colour corresponding to $X$, and note that $M'_X$ is connected in this colour. Then all edges in  $G-V'_i-V(M'_X)$ must have the other colour, and we can easily cover all but at most $4\eps n$ vertices of $G-V(M'_X)$ with a connected matching in this colour.

So, from now on, let us assume that there there is no choice of $X\in\{R,B\}$ and $i\in\{1,2\}$ as above. That is, we assume
\begin{equation}\label{plata}
\text{$|V_i\setminus V_X|>2\eps n$
for all $X\in\{R,B\}$ and $i\in\{1,2\}$.}
\end{equation}

Set $V_\eps:=V(G)\setminus (V_R\cup V_B).$ Since there can be no edges from $x$ to $V_\eps\cap V_1$, we have that 
\begin{equation}\label{oroII}
|V_\eps\cap V_1|\leq\eps n.
\end{equation}
 
Moreover,
\begin{equation}\label{oro}
\text{if $V_R\cap V_B\cap V_1\neq\emptyset$, then $|V_\eps\cap V_2|\leq\eps n$. }
\end{equation}

Since there is no edge between $V_R\setminus V_B$ and $V_B\setminus V_R$, there are $j\in\{1,2\}$ and $Z\in\{R,B\}$ such that $|V_j\cap (V_Z\setminus V_Y)|\leq \eps n$, where $Y\in\{R,B\}$ is such that $Y\neq Z$. In other words, 
\begin{equation*}\label{Vj}
{|V_j\setminus (V_\eps\cup V_Y)|}\leq \eps n .
\end{equation*} 

Together with~\eqref{plata}, this implies that
$|V_\eps\cap V_j|>\eps n$.
 So by~\eqref{oroII}, we have that $j=2$, and by~\eqref{oro}, we know that $V_R\cap V_B\cap V_1=\emptyset$. Hence,  $V_1\setminus V_\eps$ is covered by the two disjoint sets 
  $V_{1R}:=V_1\cap (V_R\setminus V_B)$ and $V_{1B}:=V_1\cap (V_B\setminus V_R)$. 
  
  By~\eqref{plata} and~\eqref{oroII}, we have that   $|V_{1R}|, |V_{1B}|> \eps n$. Now,  as  no edge may exist between $V_{1R}$ and $V_B\setminus V_R$, or between $V_{1B}$ and $V_R\setminus V_B$, we see that $V_2\setminus V_\eps = V_R\cap V_B$. Observe that thus, all edges between $V_2\setminus V_\eps$ and $V_{1R}$ are red, and all edges between $V_2\setminus V_\eps$ and $V_{1B}$ are blue.  
 
Also, by definition of $V_\eps$, all edges between $V_\eps\cap V_2$ and $V_{1R}$ are blue, and all edges between $V_\eps\cap V_2$ and $V_{1B}$ are red. Additionally, since there are no edges between $V_\eps\cap V_1$ and $V_2\setminus V_\eps = V_R\cap V_B$, either $V_\eps\cap V_1$ is empty, or $V_2\setminus V_\eps$ has less than $\eps n$ vertices. In the first case, we have a split colouring, and in the latter case, we can easily cover all but at most $2\eps n$ vertices of $G$ with two monochromatic connected matchings of distinct colours.
\end{proof}

\section{Covering almost everything with three cycles}\label{sec:cycles}

\subsection{Regularity preliminaries}
 
 In this subsection we introduce some very well-known concepts; the reader familiar with regularity is invited to skip this.
We start giving the standard definition of regularity.
Given $\epsilon>0$ and disjoint subsets $A, B$ of the vertex set of a graph $G$, we say that the pair $(A, B)$ is $\epsilon$\textit{-regular}, and of density $d(A,B)$,  if, for every pair $(A', B')$ with $A'\subseteq A$, $|A'|\geq \epsilon |A|$, $B' \subseteq B$, $|B'|\geq \epsilon |B|$, we have
$$ 
 \left| d(A',B')-d(A,B)\right| <\epsilon.
$$
When there is no danger of confusion, we simply say that the pair $(A,B)$ is $\eps$-regular.
It is well-known and easy to see that together with a pair $(A,B)$, its complement is $\eps$-regular, and of density $1-d(A,B)$.

For a graph $G$, we say a partition $V_0\cup V_1\cup\dots\cup V_t$ of its vertex set is $\epsilon$-regular if the following hold:
\begin{itemize}
 \item [(i)]$|V_1|=|V_2|=...=|V_t|$ and $|V_0|\leq \epsilon |V(G)|$, and
 \item [(ii)] for each $i$, $1\leq i\leq t$, all but at most $\epsilon t^2$ of the pairs $(V_i, V_j)$, $1\leq j\leq t$, are  $\epsilon$-regular. 
\end{itemize}

We state Szemer\'edi's regularity lemma~\cite{Sze78} in its form with a prepartition.
We say a partition $V_1\cup V_2\cup\dots\cup V_t$ \emph{refines} another partition $W_1\cup W_2\cup\dots\cup W_s$ if for each $i$ there is a $j$ such that $V_i\subseteq W_j$.

\begin{theorem}[Regularity lemma with prepartition]\label{regu}
For every $\eps>0$ and $m_0, s\in\mathbb N$, there is an $m_1 \in \mathbb N$ such that the following holds for each graph $G$ on $n\geq m_1$ vertices,
and with a partition $W_1,\dots, W_s$ of its vertex set.
There exists an $\eps$-regular partition $V_0\cup V_1\cup\dots\cup V_t$ such that  
\begin{itemize}
\item $V_1\cup V_2\cup\dots\cup V_t$ refines $W_1\cup W_2\cup\dots\cup W_s$, and
\item $m_0\leq t\leq m_1$.
\end{itemize}
\end{theorem}

It is usual to define a reduced graph of a graph $G$, for a given $\eps$-regular partition, and a threshold density $\rho$. 
This is the graph $R_{G}=([t], E(R_{G}))$ which has an edge for each $\eps$-regular pair of density at least $\rho$. (Note that in spite of the notation~$R_G$, the reduced graph $R_G$ depends on $G$, on $\rho$, and on the given partition.)

 \subsection{Proof of Theorem~\ref{thm:cycles}}

Theorem~\ref{regu} applied with parameters $\eps\ll \delta$, $m_0=1/\eps$ and  $s=2$ and $s=3$, gives two values for $m_1$ of which we take the larger one.
Let $G$ be a fair complete multipartite graph on $n>m_1$ vertices whose edges are coloured red and blue. We assume the colouring is not $\delta$-close to a split colouring. 

Use Lemma~\ref{lem:reductiontotripartite} to obtain a spanning fair subgraph $G'\subseteq G$ that is complete tripartite or complete bipartite (possibly $G'=G$). In the case that  the smallest partition class of $G'$ has less than $\delta n/20$ vertices (and thus $G'$ is tripartite), we delete this  class, and the same number of vertices from the other classes, in a way that the obtained bipartite graph (which we still call $G'$) is balanced.
By the proof of Lemma~\ref{lem:reductiontotripartite}, we know that $|E(G)\setminus E(G')|<\delta n^2/10$. 

Now, Theorem~\ref{regu} applied to the red subgraph of $G'$ yields a partition of $V(G')$ refining the bi- or tripartition which is $\eps$-regular in both colours. Using the majority colouring of each pair (that is, we use $\rho=1/2$, and in case of a tie we give the edge any colour), we can work with a two-coloured reduced graph~$R_{G'}$. 

 Note that the non-neighbours of a vertex $v\in V(R_{G'})$ correspond to irregular pairs containing $v$.
So, there are at most $\sqrt\eps t$ vertices $v$ in $R_{G'}$ that have more than $\sqrt\eps t$ non-neighbours in the other partition class(es).
We discard these vertices from the reduced graph $R_{G'}$ to ensure that each vertex of the remaining graph $R'_{G'}$ has at most $\sqrt\eps t$ non-neighbours in the other partition class(es). 

Observe that the size of any class $C_i$ of the bi- or tripartition lies between 
\begin{equation}\label{parti-class}
\frac{|\hat C_i|-|V_0|}{|V_1|}-\sqrt\eps t\geq \frac{|\hat C_i|}{|V_1|} -2\sqrt\eps t\text{ \  \ and \ \ }\frac{|\hat C_i|}{|V_1|},
\end{equation}
 where $\hat C_i$ is the class of the bi-/tripartition of $G'$ corresponding to $C_i$.  So, since~$G'$ is fair, and since $|V_0|\leq \eps n,$ we have that $|C_1|\leq |C_2|+3\sqrt\eps t$, or $|C_1|\leq |C_2|+|C_3|+2\sqrt\eps t$, respectively, if $C_1$ is the largest partition class of $R'_{G'}$. 
If necessary, we discard at most $3\sqrt\eps t$ vertices from $C_1$ to make the remaining  graph $R''_{G'}$ fair.

Resuming, we have obtained a bi- or tripartite subgraph $R''_{G'}$ of $R_{G'}$  on  at least $(1-4\sqrt\eps )t$ and at most $t$ vertices.
Let the partition classes of $R''_{G'}$ be denoted $C_1'$, $C_2'$, and, if applicable, $C_3'$.
In $R''_{G'}$, every vertex of $C_i'$ has at most $\sqrt \eps t$ many non-neighbours outside of $C_i'$.
Since $R''_{G'}$ is fair, $t-|C_i'| \ge t/2$, and so every vertex of $C_i'$ has degree at least
\[
t-|C_i'|-\sqrt\eps t \ge (1-2\sqrt\eps)(t-|C_i'|),
\]
where we sum the degree over both colours.
We may view $R''_{G'}$ as a reduced graph for a subgraph $G''$ of $G'$, with $|V(G'')|\geq (1-5\sqrt\eps )n$ (allowing $\sqrt\eps n \geq \eps n$ for the exceptional set $V_0$). Note that 
\begin{equation}\label{fewedgeslost}
|E(G)\setminus E(G'')|\leq \delta n^2/10 + 5\sqrt\eps n^2\leq \delta n^2/5.
\end{equation}

Observe that in the case that $G'$ is bipartite, the colouring in the reduced graph $R''_{G'}$ might be a split colouring, even if the colouring in $G$ is not a split colouring. But, in this case, note that if we change the colouring of any edge of $R''_{G'}$, we no longer have a split colouring.
As all our arguments work the same whether we take the threshold for the colouring of $E(R''_{G'})$ to be $1/2$ or $\delta/2$, we may assume that we can either change the colouring of $R''_{G'}$ in a justified way so that the obtained colouring is not a split colouring, or that $R''_{G'}$ has a split colouring, and each red edge of $R''_{G'}$ corresponds to a pair in $G''$ which has blue density $<\delta/2$, and vice versa for the blue edges. So, deleting at most $\delta |E(G'')|/2$ edges from $G''$ we can make all pairs monochromatic. Thus, by~\eqref{fewedgeslost}, this means $G$ is $\delta$-close to a split colouring, a contradiction.

Therefore, we may assume the colouring in  $R''_{G'}$ is not a split colouring. Since by~\eqref{parti-class}, the classes of the bi- or tripartition of $R''_{G'}$ have size at least  $\delta t/20 - 2\sqrt\eps n$, we may apply Lemma~\ref{lem:conn-matchings-robust} or Lemma~\ref{lem:conn-matchings-robust-for-bip} to $R''_{G'}$ to obtain two connected matchings, one in each colour, which together cover all but at most $72\sqrt\eps t$ vertices.
Now we use a crucial and well-known lemma that has appeared in similar forms before. In its original form it is due to \L uczak~\cite{Luc99}. 
The version we use here is close to the one given in from Section 4 of~\cite{GRSS11}.

\begin{lemma}\label{lem_connmatchTOcycles}
Let $R$ be the reduced two-coloured graph of a two-coloured graph~$H$, for some $\gamma$-regular partition, where each edge of $R$ corresponds to a $\gamma$-regular pair of density at least $\sqrt\gamma$.\\
If all but at most $72 \gamma |V(R)|$ vertices of $R$ can be covered with two disjoint connected monochromatic matchings, one of each colour, then $H$ has two disjoint monochromatic cycles, one of each colour, which together cover at least $(1-100\sqrt\gamma)|V(H)|$ vertices of $H$.
\end{lemma}

For completeness, let us outline  a proof of Lemma~\ref{lem_connmatchTOcycles}. 

\begin{proof}[Sketch of a proof of Lemma~\ref{lem_connmatchTOcycles}]
We first connect in $H$ the pairs corresponding to matching edges with monochromatic paths, following their connections in $R$. 
We do this simultaneously for both colours. 
Note that in total, these paths consume only a constant number of vertices of~$H$.
Then we connect the monochromatic paths using the matching edges, blowing up the edges to long paths, where regularity ensures we can use all but a small fraction of the corresponding pairs. 
This gives the desired cycles.
The above argumentation is also explained, rather detailed, in the proof of the main result of~\cite{GRSS06b}.
\end{proof}

Applying Lemma~\ref{lem_connmatchTOcycles} to the graph $R''_{G'}$
gives the desired two cycles which cover all but at most $100\sqrt\eps n < \delta n$ of our graph $G$. This finishes the proof of Theorem~\ref{thm:cycles}.

\section{Covering all vertices with 14 cycles}\label{sec:exact-cycles}

\subsection{Preliminaries}

Call a balanced bipartite subgraph $H$ of an $2n$-vertex graph $\eps$-\emph{hamiltonian}, if any balanced bipartite subgraph of $H$ with at least
$(1-\eps) n$ vertices in each partition class is hamiltonian. The next lemma is a straighforward combination of results of Haxell~\cite{Hax97} and Peng, R\"odl and Ruci\'nski~\cite{PRR02}.
A proof of this lemma is given in our companion paper~\cite{LSS14}.

\begin{lemma}\cite{LSS14}
\label{lem:eps-hamiltonian}
For any $1 >\gamma > 0,$ there is an $n_0 \in
\mathbb{N}$ such that any
balanced bipartite graph $G$ on at least $2n \geq 2n_0$ vertices and of
edge density at least $\gamma$ has an $\gamma/4$-hamiltonian subgraph
of size at least $\gamma^{3024/\gamma} n/3$.
\end{lemma}

In order to absorb vertices not covered with the cycles given by Theorem~\ref{thm:cycles}, we use the following result.

\begin{lemma}[Gy\'{a}rf\'{a}s, Ruszink\'o, S\'ark\"ozy, and Szemer\'edi~\cite{GRSS06b}]
\label{lem:somewhat-unbalanced}
There is an $n_0 \in \mathbb N$ such that for $n\geq n_0$ and $m \leq \frac{n}{(8r)^{8(r+1)}}$ and for any 
colouring of the edges of $K_{n,m}$ with~$r$ colours, there are $2r$ vertex-disjoint monochromatic cycles that cover the $m$ vertices of the smaller side.
\end{lemma}

\subsection{Proof of Theorem~\ref{thm:covering-all}}

Let $G$ be a fair complete $k$-partite graph on $n\geq n_0$ vertices, whose edges are 2-coloured. We assume $n_0$ to be large enough, its value may be extracted from the proof.
By Lemma~\ref{lem:reductiontotripartite}, we may assume that $2 \le k \le 3$.
Let $V_1$, $V_2$, and possibly $V_3$, be the partition classes of $G$, where we assume that $|V_1|\ge|V_2|\ge|V_3|$.
Set $\delta := 2^{-10^4}$.

For technical reasons, we  split the argument into three cases.
The explicit value of $\delta$ plays a role only in the last case of the proof.

We first discuss the case when 
\begin{equation}\label{eq:super-fair}
|V_1| \le |V_2| + |V_3| - \delta n \mbox{ and } |V_3| \ge \delta n.
\end{equation}
We pick disjoint subsets $U_1^1,U_1^2 \subseteq V_1$, $U_2^1,U_2^3 \subseteq V_2$, and $U_3^2,U_3^3 \subseteq V_3$ of size $\lfloor \delta n / 4 \rfloor$ each.
Due to~\eqref{eq:super-fair}, it holds that $G - \{U_i^j : i,j \in \{1,2,3\}\}$ is fair.
Indeed, every graph $H$ with $G - \{U_i^j : i,j \in \{1,2,3\}\} \subseteq H \subseteq G$ is fair, a fact that we will exploit later. 

Let $G_1 := G[U_1^1, U_2^1]$, $G_2 := [U_1^2, U_3^2]$, and $G_3 := G[U_2^3, U_3^3]$. Assuming $n_0$ is large enough,
we apply Lemma~\ref{lem:eps-hamiltonian} with $\gamma:=1/2$ to $G_i$, $i=1,2,3$, considering only the edges of the respective majority colour. This gives  monochromatic $1/8$-hamiltonian (and thus balanced) subgraphs $H_i=[W^i_1, W^i_2]$ of $G_i$, $i=1,2,3$ with
\begin{equation}\label{eqn:orderofHi}
|W^i_1|  \ge {\delta n}/{2^{6053}}.
\end{equation}

Now let $H := G - \bigcup_{i=1}^3 V(H_i)$ and recall that $H$ is fair since $G - \{U_i^j : i,j \in \{1,2,3\}\} \subseteq H \subseteq G$.
By~\eqref{eq:super-fair} and by the choice of the sets $U^i_j$, each partition class of $H$ has size at least $\delta n / 2$. Let $\delta':={\delta }/{2^{6148}}$.
Assuming $n_0$ is large enough, we know $H$ is not $\delta'$-close to a split colouring, and Theorem~\ref{thm:cycles} gives that all but at most $\delta' n$ vertices of $H$ can be partitioned into two monochromatic cycles.
If the set $X \subseteq V(G')$ of uncovered vertices has odd cardinality, we cover one of its vertices with a trivial cycle, and from now on assume $|X|$ is even.

Assume $|X \cap V_3| \ge |X \cap V_1|,|X \cap V_2|$ (the other cases are similar).
Partition $X$ into equal-sized sets $X_1,X_2$ with 
$X \cap V_i \subseteq X_i$ for $i=1,2$.
Assuming $n_0$ is sufficiently large, the choice of $\delta'$ and~\eqref{eqn:orderofHi} give that 
\begin{equation}\label{eq:eq:eq}
|X_1|  \le {\delta'n}/{2} \le {|W^1_2|}/{16^{24}}.
\end{equation}
 So, Lemma~\ref{lem:somewhat-unbalanced} yields a set $\mathcal C$ of eight vertex-disjoint monochromatic cycles in $[X_1,W^1_2]$ and $[X_2,W^1_1]$ that together cover $X_1 \cup X_2$.
As we may assume that $\mathcal C$ does not contain any trivial cycles,  $V(\bigcup \mathcal C)\cap V(H_1)$ is balanced, and each side has at most $ {|W^1_2|}/{16^{24}}$ vertices, by~\eqref{eq:eq:eq}.
Since $H_1$ is $(1/8)$-hamiltonian, the monochromatic graph $H_1-V(\bigcup \mathcal C)$ thus has a hamilton cycle.
Also, $H_2$ and $H_3$ each admit a hamilton cycle.
In total, we covered $V(G)$ using at most $2+1+8+3=14$ vertex-disjoint monochromatic cycles.
\medskip

We now discuss the case when $k=2$.
We use the same method as above. 
Observe that now, we only need to find one $(1/8)$-hamiltonian subgraph of $G$ (as opposed to the three graphs $H_1,H_2,H_3$ above).  
Thus in the last step we will use  only one instead of three hamilton cycles. 
However, we might need three cycles to cover $G$ almost entirely, since the colouring of $G$ might be close to a split-colouring (a covering with 3 cycles does exist by the remark after Theorem~\ref{thm:cycles}).
Since we may assume that none of these three cycles is trivial, we may assume the set $X$ of vertices not covered with the three cycles to have even cardinality.
In this way we avoid another cycle. 
This means we can cover $G$ with $3+8+1=12$ vertex-disjoint monochromatic cycles in total.
\medskip

To complete the proof of Theorem~\ref{thm:covering-all}, we now consider the case when $k=3$, but~\eqref{eq:super-fair} is violated.
Essentially, we proceed as in the case of $k=2$.
If $n$ is odd, we pick a single vertex $v$ from $V_1$ as a cycle.
Since $G-v$ is fair, we may simply assume that $n$ is even.

No matter whether we have $|V_3| < \delta n$, or we have $|V_3| \ge \delta n$ and $|V_1| > |V_2| + |V_3| - \delta n$, we proceed as follows. Delete a set $W$ of $|V_2| + |V_3|- |V_1| < \delta n$ vertices from $V_3$. Consider the bipartite graph $G'$ spanned by $V_1$ and the remains of $V_2\cup V_3$, which we call $V'_2$. Note that $G'$ is complete bipartite and balanced.

Let $U_1 \subseteq V_1$ and $U_2 \subseteq V_2$ have size $\lceil n/4\rceil$ each.
(Note that clearly $|V_2|\geq n/4$.) Assuming $n_0$ is large enough, we may
apply Lemma~\ref{lem:eps-hamiltonian} with $\gamma=1/2$ to the graph induced by  the majority colour to find a $(1/8)$-hamiltonian subgraph $H=[U'_1,U'_2]$ of $G'[U_1 \cup U_2]$, with 
$|U'_1|=|U'_2| \ge n/2^{6052}$.
Assuming $n_0$ is large enough, the remark after
Theorem~\ref{thm:cycles} gives three monochromatic vertex-disjoint cycles that cover all but a set $W'$ of at most $\delta n$ many vertices of $G'- V(H)$.
We may assume that none of these cycles is trivial, and thus $|W' \cap V_1| = |W' \cap V_2'|$.

Partition $W \cup W'$ into two equal-sized sets $W_1$ and $W_2$ such that for $i=1,2$, all edges between $W_i$ and $U_i'$ are present. As
$$16^{24} |W_1|  < \delta  n /2^{97}\leq n/2^{6052} \le |U_2'|,$$
we may apply Lemma~\ref{lem:somewhat-unbalanced} to cover each of $W_1$, $ W_2$ with four vertex-disjoint monochromatic cycles, which we again assume to be non-trivial.
Since $H$ is $(1/8)$-hamiltonian, the yet unused vertices of $H$ span a hamiltonian cycle. In total, we covered $G$ with $1+3+8+1=13$ vertex-disjoint monochromatic cycles.

\section{Conclusion}\label{sec:conclu}

Our Theorem~\ref{thm:mainresult} and the earlier result of Pokrovskiy together cover all cases of monochromatic path covers in $2$-edge-coloured multipartite graphs. 
It would be natural to investigate the same problem for more colours, in the spirit of Conjecture~\ref{con:gyarfas}. 
For bipartite graphs whose edges are coloured with $r$ colours, Pokrovskiy~\cite{Pok14} conjectures that there is a partition into $2r-1$ monochromatic paths, and he shows this number is best possible. 
For $k$-partite graphs with $k\geq 3$ a partition into less monochromatic paths, perhaps $r$ paths, might be possible, in analogy to the case $r=2$ we treated here.

\medskip

For cycle partitions, as said in the introduction, we believe that Theorem~\ref{thm:cycles} is not best possible. 
However, recalling the example given in the introduction, a split colouring between a partition class of size $n/2$ and the rest of the graph,
it is not always possible to cover {\it all} vertices of any large fair multipartite graph with two vertex-disjoint monochromatic cycles.
This remains true even if the colouring of this graph is far from a split colouring (the other classes might have size $n/4$ each). 

Even if a partition class containing half of the vertices of the graph is forbidden, such constructions are possible.
Indeed, given a properly split-coloured balanced complete bipartite graph one can add a third partition class consisting of a single vertex $v_3$ seeing only one colour. Then the obtained graph cannot be partitioned into two monochromatic cycles. One can even add a fourth partition class consisting of a single vertex $v_4$ seeing only the other colour, while giving $v_3v_4$ any colour. Still the obtained graph has no partition into two monochromatic cycles.

\begin{problem}\label{prob:cyles}
For $k\geq 3$, under which conditions does a fair complete $k$-partite graph with a 2-colouring of the edges admit a partition into two monochromatic cycles?
\end{problem}

One candidate for a sufficient condition in Problem~\ref{prob:cyles} could be balancedness, that is, having partition classes of equal size.
In any case, we think that considering balanced multipartite graphs is a reasonable restriction that might be of its own interest.

\medskip

The next natural step is to extend Theorems~\ref{thm:pokrovskiy} and~\ref{thm:cycles} to more colours. We have seen here that three disjoint monochromatic cycles can cover all but $o(n)$ of the vertices of any large enough multipartite graph, whose edges are two-coloured (where three cycles are only needed if the colouring is very close to a split colouring). 
In~\cite{LSS14}, together with Richard Lang, we prove that in  large enough  3-edge coloured balanced complete bipartite graphs,  five disjoint monochromatic  cycles suffice to cover the graph almost entirely. 

In analogy to Pokrovskiy's conjecture for path covers~\cite{Pok14} mentioned above, it might always be possible to cover almost all vertices of any large enough multipartite graph, whose edges are $r$-coloured, with $2r-1$ disjoint monochromatic cycles. Perhaps the number of cycles needed can even be dropped to  $r$ in $k$-partite graphs with $k\geq 3$.
Maybe this is even possible in bipartite graphs with a colouring sufficiently far from a specific problematic colouring.


\bibliographystyle{amsplain}
\bibliography{monochromatic-cover}

\end{document}